\documentclass[a4paper,11pt]{article}

\usepackage[margin=3cm]{geometry}
\usepackage{setspace}
\usepackage{amsmath}
\usepackage{amssymb}
\usepackage{amsthm}

\usepackage{psfrag}
\usepackage{graphicx,subfigure}
\usepackage{color}

\usepackage{alltt}

\usepackage{enumerate}

\def\bG{\bar{\Gamma}}
\def\bS{\bar{\Sigma}}
\def\bV{\bar{V}}
\def\bU{\bar{U}}

\def\tg{\tilde{\gamma}}
\def\ts{\tilde{\sigma}}
\def\tv{\tilde{v}}
\def\tu{\tilde{u}}
\def\del{\partial}
\def\eps{\varepsilon}

\newtheorem{lemma}{Lemma}[section]
\newtheorem{proposition}{Proposition}[section]
\newtheorem{remark}{Remark}[section]

\begin{document}
\title{Localization in inelastic rate dependent \\
shearing deformations
}
\author{Theodoros Katsaounis\footnotemark[1]\ \footnotemark[2]
\and Min-Gi Lee\footnotemark[1]
\and Athanasios Tzavaras\footnotemark[1]\  \footnotemark[3]  \footnotemark[4]}
\date{}

\maketitle
\renewcommand{\thefootnote}{\fnsymbol{footnote}}
\footnotetext[1]{Computer, Electrical and Mathematical Sciences \& Engineering Division, King Abdullah University of Science and Technology (KAUST), Thuwal, Saudi Arabia}
\footnotetext[2]{Department of Mathematics and Applied Mathematics, University of Crete, Heraklion, Greece}
\footnotetext[3]{Institute of Applied and Computational Mathematics, FORTH, Heraklion, Greece}
\footnotetext[4]{Corresponding author : \texttt{athanasios.tzavaras@kaust.edu.sa}}
\renewcommand{\thefootnote}{\arabic{footnote}}

\baselineskip=12pt

\begin{abstract}
 Metals  deformed at high strain rates can exhibit failure through formation of shear bands, a phenomenon often attributed 
to Hadamard instability and localization of the strain into an emerging coherent structure. 
We verify  formation of shear bands for a nonlinear model exhibiting strain softening and strain rate sensitivity. The effects of strain softening and strain rate sensitivity 
are first assessed by linearized analysis, indicating that the combined effect leads to Turing instability. 
For the  nonlinear model a class of  self-similar solutions is constructed, that depicts a coherent localizing structure and the formation of a shear band. 
This solution is associated to a heteroclinic orbit of a dynamical system. The orbit is constructed numerically and yields explicit shear localizing solutions. 
\end{abstract}



\section{Introduction}
Shear bands are narrow zones of intense shear observed during the dynamic deformation of many metals at high strain rates. 
{  Shear localization
forms a striking instance  of material instability, often preceding rupture, and its study has attracted considerable  attention in the mechanics  {\it e.g.}~\cite{clifton_critical_1984, FM87, MC87,AKS87, WW87, clifton_rev_1990, tzavaras_nonlinear_1992,  BD02, wright_survey_2002},
numerical \cite{walter_1992,CB99, estep_2001},
or mathematical literature \cite{dafermos_adiabatic_1983, tzavaras_plastic_1986, bertsch_effect_1991, estep_2001,KT09}.
In experimental investigations of high strain-rate deformations of steels,  observations of shear bands are  typically associated 
with strain softening response -- past a critical strain -- of the measured stress-strain curve \cite{clifton_critical_1984}. 
It was proposed by Zener and Hollomon \cite{zener_effect_1944}, and further precised by Clifton et al  \cite{clifton_critical_1984, SC89},
that the effect of the deformation speed is twofold:
An increase in the deformation speed changes the deformation
conditions from isothermal to nearly adiabatic, and the combined effect of thermal softening and strain hardening of metals may produce 
a net softening response. On the other hand, strain-rate hardening has an effect {\it per se}, inducing momentum diffusion and playing a stabilizing role.

Strain softening has a destabilizing effect; it is known at the level of simple models ({\it e.g.} Wu and Freund \cite{WF84}) to induce \emph{Hadamard instability} - 
an ill-posedness of a linearized problem.
It should however be remarked that, while Hadamard instability indicates the catastrophic growth 
of oscillations around a mean state,  what is observed in localization is the orderly albeit extremely fast development of coherent structures, the shear bands. 
Despite considerable  attention to the problem of localization, little is known about the initial formation of shear bands, 
due to the dominance of nonlinear effects from the early instances of localization. }

Our aim is to study the onset of localization for a simple model, 
lying at the core of various theories for  shear band formation,
\begin{equation}
\label{intro-maineq}
\begin{aligned}
\del_t v   &=  \del_x  \big (  \varphi (\gamma) \, v_x^n \big ),
\\
\del_t \gamma  &= \del_x v  \, ,
\end{aligned}
\end{equation}
which will serve to assess the effects of strain softening $\varphi' (\gamma)  < 0$ and strain-rate sensitivity $0<n\ll1$ and 
analyze the emergence of a shear band out of the competition
of Hadamard instability and strain-rate hardening.
The model \eqref{intro-maineq} describes shear deformations of a viscoplastic material in the $xy$-plane, with $v$ the velocity in the $y$-direction, 
$\gamma$ the plastic shear strain  (elastic effects are neglected), and the material is obeying a viscoplastic constitutive law of power law type,
\begin{equation}
\label{intro-stress}
\sigma = \frac{1}{\gamma^m} (\gamma_t)^n  \, ,  \quad  \; \; \mbox{corresponding to \; $\varphi(\gamma) = \gamma^{-m}$ with $m >0$} \, ,
\end{equation}
The constitutive law \eqref{intro-stress} can be thought as describing a plastic flow rule on the yield surface. 
The model \eqref{intro-maineq} captures the bare essentials of the localization mechanism proposed in \cite{zener_effect_1944, clifton_critical_1984, SC89}.
Early studies of \eqref{intro-maineq} appear in Hutchinson and Neal \cite{HN77} (in connection to necking),  Wu and Freund \cite{WF84} (for linear rate-sensitivity)  
and  Tzavaras \cite{tzavaras_plastic_1986, tzavaras_strain_1990, tzavaras_nonlinear_1992}.

The uniform shearing solutions,  
\begin{equation} \label{intro-uss}
 v_s(x) = x, \quad 
 \gamma_s(t) = t + \gamma_0 \, , \quad  
 \sigma_s(t) = \varphi (  t + \gamma_0 ) \, , 
\end{equation}
form a universal class of solutions to \eqref{intro-maineq} for any $n \ge 0$.  When $n = 0$ and $\varphi'(\gamma) < 0$ the system \eqref{intro-maineq}
is elliptic in the $t$-direction and the initial value problem for the associated linearized equation presents Hadamard instability; nevertheless it 
admits \eqref{intro-uss} as a special solution. Rate sensitivity $n > 0$ offers a regularizing mechanism, and the associated system \eqref{intro-maineq}
belongs to the class of hyperbolic-parabolic systems ({\it e.g.} \cite{tzavaras_strain_1990}).   
{   The linearized stability analysis of \eqref{intro-uss} has been studied in \cite{FM87,MC87,SC89} for (even more complicated models including) 
\eqref{intro-maineq}-\eqref{intro-stress}.
As  \eqref{intro-uss} is time dependent,  the problem of linearized stability leads to the study of non-autonomous
linearized problems. This was addressed by Fressengeas-Molinari \cite{FM87} and  Molinari-Clifton \cite{MC87} who introduced
the study  of relative perturbations, namely to assess the stability of the ratios of the perturbation relative to the base time-dependent solution,
and provided linearized stability results.
Such linearized stability and instability results compared well with studies of nonlinear stability ({\it e.g}  \cite{tzavaras_plastic_1986} or \cite{tzavaras_nonlinear_1992} for a survey).
}

Here, we restrict to the constitutive function $\varphi(\gamma) = \frac{1}{\gamma}$  in  \eqref{intro-stress} but retain the dependence in $n$,
and study
\begin{align}
 v_t = \bigg( \frac{v_x^n}{ \gamma } \bigg)_x \, ,  \quad 
 \gamma_t = v_x \, .
 \label{intro-model}
\end{align}
This model has a special and quite appealing  property:  After considering a transformation to relative perturbations and a rescaling of variables,
\begin{equation}
\label{intro-rescale}
\begin{aligned}
v_x (x,t)  =: u (x,t) =  U (x,\tau(t))  \, ,  \quad 
 \gamma(x,t) &= \gamma_s(t) \, \Gamma(x, \tau(t)) \, , \quad 
 \sigma(x,t) = \sigma_s(t) \, \Sigma(x, \tau(t)) 
 \\
 \mbox{where} \quad 
 &\tau(t) = \log(1+ \frac{t}{\gamma_0}) \, ,
\end{aligned}
\end{equation}
the problem of stability of the time-dependent uniform shearing solution \eqref{intro-uss} is transformed into the problem
of stability of the equilibrium $( \bar U , \bar \Gamma ) = (1,1)$ for the nonlinear but  {\it autonomous} parabolic system 
\begin{align}
 U_\tau = \Sigma_{xx} =\bigg( \frac{U^n}{\Gamma}\bigg)_{xx} \, , \quad  \Gamma_\tau = U - \Gamma \, .
 \label{intro-redsys}
\end{align}

{   The following heuristic argument leads to a conjecture regarding the effect of rate sensitivity $n$ on the dynamics: As time proceeds the second equation in \eqref{intro-redsys},
which is of relaxation type,  relaxes to the equilibrium manifold $\{ U = \Gamma \}$. Accordingly, the stability of \eqref{intro-redsys} is determined by the equation
describing the effective equation
\begin{equation*}
U_\tau = \big(  U^{n-1} \big)_{xx} \, .
\end{equation*}
The latter is parabolic for $n >1$ and backward parabolic for $n < 1$, what suggests instability in the range $n < 1$.
The argument is proposed in \cite{KT09} in connection 
to the development  of an asymptotic criterion for the quantitative assessment of  shear band formation 
and can be quantified by means of  an  asymptotic expansion (see Section \ref{sec:coarse_scale}).

In this article we study the dynamics of the system \eqref{intro-redsys}. 
First, we provide a complete analysis of linearized stability. 
The linearized system around the equilibrium $(1,1)$  takes the form
\begin{equation}
\label{intro-linp}
\begin{aligned}
{\tilde{U}}_\tau &= n \tilde{U}_{xx}  - \tilde\Gamma_{xx}, 
\\
{ \tilde\Gamma }_\tau &= \tilde{U} - \tilde\Gamma,  
\end{aligned}
\end{equation}
and is simple to analyze via Fourier analysis.  A complete picture emerges:
\begin{itemize}
\item[(a)]  For $n=0$, high-frequency modes grow exponentially fast and indicate catastrophic growth and Hadamard instability.
   \item[(b)] For $0 < n < 1$, the modes still grow and are unstable  but at a tame growth rate. 
\item[(c)] For $n > 1$ strain-rate dependence is strong and stabilizes the motion.
\end{itemize}
The instability occurring in rate-dependent localization resembles at the linearized level to the {\it Turing instability} 
familiar from problems of morphogenesis \cite{Turing52} ({\it cf.} Remark \ref{remarkTuring}).
}

Next,  we turn to the nonlinear system \eqref{intro-redsys} and proceed to analyze 
the competition between Hadamard instability
and strain-rate dependence in the nonlinear regime in the parameter range $0 < n < 1$. 
Exploiting scaling properties of \eqref{intro-redsys} we construct 
a class of {\it focusing} self-similar solutions of  the form 
\begin{equation}
 \begin{aligned}
 \gamma(x,t) &= \gamma_0\Big(1 + \frac{t}{ \gamma_0 } \Big)^{1+ \frac{ 2 \lambda}{2-n} } \;\bG\bigg(x\Big(1 + \frac{t}{ \gamma_0 } \Big)^{\lambda}\bigg), \\
 \sigma(x,t) &= \frac{1}{\gamma_0}\Big(1 + \frac{t}{ \gamma_0 } \Big)^{ -1- 2 \lambda \frac{1 - n}{2-n} } \;\bS\bigg(x\Big(1 + \frac{t}{ \gamma_0 } \Big)^{\lambda}\bigg), \\
u(x,t) &=   v_x (x,t)  = \Big(1 + \frac{t}{ \gamma_0 }\Big)^{\frac{2 \lambda}{2-n} } \;\bU\bigg(x\Big(1 + \frac{t}{ \gamma_0 } \Big)^{\lambda}\bigg). \\
 \end{aligned} \label{intro-ss}
\end{equation}
{    For $\lambda > 0$ -- and as opposed with the usual self-similar solutions of diffusion equations -- the information will
propagate on the lines $x\big(1 + \frac{t}{ \gamma_0 } \big )^\lambda = const$ and  focus towards the center $x = 0$.
The solution \eqref{intro-ss}  depends on four-parameters: $n$, the growth-rate $\lambda > 0$, and the initial data  $(\bG_0, \bU_0)$ standing for the sizes of the
initial nonuniformities:  $\bG_0=\bG(0)$ in the strain and  $\bU_0=\bU(0)$ in the strain rate. 
The profiles $(\bar \Gamma, \bar U,  \bar \Sigma)$, $\bar U = \bar V_\xi$,   solve the singular system \eqref{intro-barsys} and are constructed 
numerically. For the construction, it turns out that the parameters need to satisfy the constraints
\begin{equation}
\label{intro-size}
\lambda = \frac{2-n}{2}\left(\frac{\bU_0}{\bG_0} -1 \right) \, , \qquad 0< \lambda < \frac{(2-n)(1-n)}{n} \, .
\end{equation}
The uniform shear solution corresponds to the choice $\bar U_0 = \bar \Gamma_0 =1$ associated to no growth $\lambda = 0$. 
The solution \eqref{intro-ss} shows that rate-sensitivity suppresses, at the nonlinear level, the oscillations 
resulting from Hadamard instability and  that the combined process leads  into a single ``runaway"
of concentrated strain  that appears like the shear bands observed in experiments. 
It is sketched at various time instances  in Figure \ref{fig:gamma} for the strain $\gamma$, 
 in Figure \ref{fig:u} for the strain rate $\gamma_t$, in Figure \ref{fig:v} for the velocity $v$, and in Figure \ref{fig:sigma}  for the stress $\sigma$. 
The figure for the stress provides an analytical justification of the phenomenon of  stress collapse
across the shear band, predicted in theoretical results of \cite{tzavaras87} and in numerical computations of Wright and Walter \cite{WW87}. 
}

Our analysis  validates for the model \eqref{intro-model}  the onset of localization predicted by the asymptotic
analysis in \cite{KT09} along the lines of the  Chapman-Enskog expansion of kinetic theory. 
The result complements an earlier study of similar behavior
for a thermally softening temperature dependent non-Newtonian fluid \cite{KOT14}. In a companion article \cite{LT16}, we provide an
existence  proof for the heteroclinic orbit (computed numerically here);
this is accomplished  even at the level of the plastic flow rule \eqref{intro-stress} in the parameter range $0 < n < m$.  (This region
is optimal as it is known that for $m > n$ the uniform shear is asymptotically stable \cite{tzavaras_strain_1990}.) The proof of existence for the heteroclinic
employs the geometric singular perturbation theory for dynamical systems and is outside the scope of the present work. 

The article is organized as follows :  in Section \ref{sec:model} we introduce the mathematical model along with its basic properties. 
The idea of relative perturbations for the stability of the time-dependent uniform shearing solutions is reviewed in Section \ref{sec:relative} and the complete
 linearized stability analysis is  presented in  Section \ref{sec:linear}. 
A dichotomy of stability for the linearized problem appears, depending on the strain-rate sensitivity, in accord with \cite{FM87,MC87} and 
results on nonlinear stability \cite{tzavaras_plastic_1986, tzavaras_nonlinear_1992}.
From Section \ref{sec:coarse_scale} onwards, we study nonlinear effects. First, an effective equation is derived via an asymptotic analysis 
following \cite{KT09} that postulates  instability in the parameter regime $0 < n < 1$. 
The emergence of localization is studied in Section \ref{sec:localization}: 
We introduce an ansatz of {\it focusing} self-similar solutions
\begin{equation}
  V(x,\tau) = e^{\lambda  \frac{n}{2-n} \tau} \,\bar{V}\big( \xi\big) \, , \quad 
  \Gamma(x,\tau) = e^{\lambda  \frac{2}{2-n} \tau}  \,\bar\Gamma\big( \xi\big) \, , \quad 
 \Sigma(x,\tau) =   e^{\lambda   \big ( -1+\frac{n}{2-n} \big )  \tau}  \,\bar\Sigma\big( \xi\big) \, , 
\label{intro-sstrans}
\end{equation}
where $\bV_\xi = \bU$,  $\xi = x e^{\lambda \tau}$ and   $\lambda >0$ is a parameter.  
Their existence is based on constructing a solution $( \bar{V} , \bar{ \Gamma} , \bar{ \Sigma } )$
for the nonlinear system of singular ordinary differential equations
\begin{equation}
 \lambda \left( \frac{n}{2-n} \bar{V} + \xi \bar{V}_\xi\right) =\bar{\Sigma}_\xi, \, , \quad 
 \lambda \left( \frac{2}{2-n} \bar{ \Gamma} + \xi \bar{ \Gamma}_\xi\right) = \bar{V}_\xi - \bar{ \Gamma } \, , \quad 
 \bar{ \Sigma } = \frac{\bar{V}_\xi^n}{\bar{\Gamma}} \, .
 \label{intro-barsys}
\end{equation}
Such singular systems may (or may not) have solutions and this is determined 
by a case-by-case analysis. One can remarkably de-singularize  \eqref{intro-barsys} and convert the problem of existence of localizing profiles
to that of constructing a suitable heteroclinic orbit for a system of ordinary differential equations (see \eqref{eq:p}-\eqref{eq:r}).
Using a combination of dynamical systems ideas and numerical computation,  we numerically construct the 
heteroclinic connection and show that it  gives rise to a coherent localizing structure. 
The heteroclinic orbit is represented by the red dotted line in Figure \ref{fig:2manifold}.  {  The properties of the localizing solutions are
summarized in Section \ref{sec:localiz}.
}

\section{Description of the model} 
\label{sec:model}

The simplest model for analyzing the dynamics of  shear band formation is the one-dimensional shear deformation of a viscoplastic material that
exhibits strain softening and strain-rate sensitivity. The  motion of the specimen occurs in the $y$-direction, with 
shear direction that of the $x$-axis, and is described by the  \emph{velocity} $ v(x,t)$, \emph{plastic strain} $\gamma(x,t)$ and \emph{stress} $\sigma(x,t)$. 
These field variables satisfy the balance of linear momentum and  the kinematic compatibility equation
\begin{align}
 v_t &= \sigma_x,   \label{eq:v0}
 \\
 \gamma_t &= v_x,   \label{eq:g0}
\end{align}
respectively.
In the simplest situation,  the elastic effects are neglected and one focusses on a viscoplastic model, where the stress
depends only on  the (plastic) strain $\gamma$ and the strain rate $\gamma_t$,
$$
\sigma = f( \gamma, \gamma_t). 
$$
The material  exhibits strain softening when $\displaystyle \frac{\partial f}{ \partial\gamma} < 0$.
A simple constitutive law of that form is
\begin{equation}
\label{softcl}
\sigma = \varphi (\gamma)  (\gamma_t )^n, 
\end{equation}
where $\varphi' (\gamma) < 0$ for strain softening, while the strain-rate sensitivity parameter $n > 0$ is thought
as very small $n \ll 1$.

\medskip
\noindent
{\bf Uniform shearing solutions}. 
The system  \eqref{eq:v0}-\eqref{eq:g0} with \eqref{softcl} admits a special class of solutions describing  uniform shearing, 
\begin{equation} \label{eq:uss}
\begin{split}
 v_s(x) &= x, \quad  u_s (t) =  (\del_x  v_s )  (x,t)  = 1, \\
 \gamma_s (t) &= t + \gamma_0, \\
 \sigma_s (t) &= \varphi\big(t + \gamma_0\big).
\end{split}
\end{equation}
Due to the strain-softening assumption $\varphi' (\gamma) < 0$, the system \eqref{eq:v0}-\eqref{eq:g0}, \eqref{softcl}, with $n = 0$
is an elliptic initial-value problem which is ill-posed exhibiting Hadamard instability. Nevertheless, both the system with $n = 0$ and its
regularized version with $n > 0$ admit the class of the uniform shearing solutions \eqref{eq:uss}. 

\medskip

Our goal is to study the stability  of the uniform shearing solutions in both cases $n = 0$ and $n > 0$.
For the remainder of this work we focus on the particular choice of $\varphi (\gamma) = \frac{1}{\gamma}$, namely
\begin{equation}
\label{specialcl}
 \sigma = \varphi (\gamma)  (\gamma_t )^n  = \gamma^{-1} \gamma_t^n , \qquad n>0.
\end{equation}
The reason for this restriction is the following: The uniform shearing solutions are time-dependent and their analysis (linearization
and nonlinear analysis) leads very quickly to issues with non-autonomous problems. 
The special choice of the constitutive relation \eqref{specialcl} has the property that it leads to autonomous problems for 
its relative perturbation  (see system \eqref{eq:sys1U}-\eqref{eq:sys1G}) and appears to be indicative of the general response in
the unstable regime.

With the choice \eqref{specialcl} system \eqref{eq:v0}-\eqref{eq:g0} reads 
\begin{align}
 v_t &= \bigg( \frac{v_x^n}{ \gamma } \bigg)_x, \label{eq:mombal} \\
 \gamma_t &= v_x.  \label{eq:kinematic}
\end{align}
An initial-boundary problem for \eqref{eq:mombal}-\eqref{eq:kinematic} is considered in $[0,1] \times \mathbb{R}^+$ with the following initial and boundary conditions
\begin{equation}
\label{eq:init_bdry}
\begin{aligned}
 v(x,0) &= v_0(x), \quad \gamma(x,0) = \gamma_0(x), 
 \\
v(0,t) &= 0, \quad v(1,t) = 1.
\end{aligned}
\end{equation}
The boundary condition reflects  imposed boundary shear.
As a consequence of the boundary conditions we have
\begin{equation}
\label{iconstraint}
 \int_0^1 v_x(y,t)\, dy=1 , \quad \text{for all $t>0$.}
\end{equation}

An equivalent formulation of  \eqref{eq:mombal}-\eqref{eq:kinematic} and \eqref{eq:init_bdry}  is obtained by considering the strain rate as the primary 
field variable (replacing the velocity). Introducing the strain rate  $ u = \gamma_t = v_x, $  the system becomes 
\begin{align}
 u_t &= \bigg( \frac{u^n}{ \gamma } \bigg)_{xx}, \label{eq:mombal2} \\
 \gamma_t &= u  . \label{eq:kinematic2}
\end{align}
The corresponding initial conditions are
\begin{align}
 u(x,0) = u_0(x), \quad \gamma(x,0) = \gamma_0(x). \label{eq:init2}
\end{align}
Note that the boundary conditions on $v$ and the compatibility condition \eqref{iconstraint} imply that
\begin{align}
 & \bigg( \frac{u^n}{ \gamma } \bigg)_x(0,t) = 0, \quad \bigg( \frac{u^n}{ \gamma } \bigg)_x(1,t) = 0,  \label{eq:bdry2} \\
 & \int_0^1 u(x,t) \, dx = 1, \quad \text{for all $t>0$.} \label{eq:constraint}
\end{align}
The resulting initial-boundary problem consists of \eqref{eq:mombal2}-\eqref{eq:kinematic2} subject to  \eqref{eq:init2}-\eqref{eq:bdry2}. The constraint \eqref{eq:constraint}
is inherited from the initial data due to the conservation of $\int_0^1 u\, dx$. It is apparent from \eqref{eq:mombal2} that there is a diffusion mechanism
in the strain-rate, manifesting the effect of strain-rate sensitivity.

\section{Stability Analysis} \label{sec:stability}

\subsection{Relative Perturbations} \label{sec:relative}

{  Motivated by the form of the uniform shearing solutions \eqref{eq:uss} and \cite{FM87,MC87}, one may introduce a rescaling of the dependent variables 
and time in the following form:}
\begin{equation}
\label{rescale}
\begin{aligned}
u(x,t) &= u_s(x) \, U(x,\tau(t)) = U(x, \tau(t)), \\
 \gamma(x,t) &= \gamma_s(t) \, \Gamma(x, \tau(t)) = \big( t + \gamma_0\big)\Gamma(x, \tau(t)), \\
 \sigma(x,t) &= \sigma_s(t) \, \Sigma(x, \tau(t)) = \big( \frac{1}{t + \gamma_0}\big)\Sigma(x, \tau(t)), \\
\dot{\tau}(t) &= \frac{1}{t + \gamma_0}, \quad\tau(0) = 0 \quad\Longrightarrow \quad \tau(t) = \log(1+ \frac{t}{\gamma_0}).
\end{aligned}
\end{equation}
In \eqref{rescale} we select
\begin{equation}
\label{defgamma0}
\gamma_0 := \int_0^1 \gamma(x,0) \, dx
\end{equation}
which normalizes the initial relative perturbation
\begin{equation}
\label{initgamma}
\int_0^1 \Gamma (x, 0 ) \, dx = 1 \, .
\end{equation}

From  \eqref{eq:mombal2}-\eqref{eq:kinematic2} we see that the new field variables satisfy,
\begin{align}
 U_\tau &= \Sigma_{xx} =\bigg(\frac{U^n}{\Gamma}\bigg)_{xx}, \label{eq:sys1U}\\
 \Gamma_\tau &= U- \Gamma. \label{eq:sys1G} 
\end{align}
The boundary condition \eqref{eq:bdry2} together with  \eqref{eq:constraint}, \eqref{initgamma} and the equations \eqref{rescale}, \eqref{eq:sys1U},
\eqref{eq:sys1G} imply the restrictions
\begin{equation}
\label{constraintUG}
\int_0^1 U (x, t) \, dx = \int_0^1 \Gamma (x,t) \, dx = 1  \quad \mbox{for $t > 0$} .
\end{equation}
Note that under the transformation \eqref{rescale} the uniform shearing motion   is transformed 
to an equilibrium
\begin{equation}
 U\equiv1, \quad \Gamma\equiv 1, \quad \Sigma\equiv 1, \label{eq:rel_uniform}
\end{equation}
 for the transformed problem \eqref{eq:sys1U}-\eqref{eq:sys1G}. Despite the fact the solution \eqref{eq:uss} is 
 time dependent, the transformed problem is still autonomous. In fact, the latter property is the reason for restricting
 to the constitutive class \eqref{specialcl}, so that spectral analysis can be  used to obtain information for the linearized problem.

\subsection{Linearized stability analysis}  \label{sec:linear}
In order to assess the growth or decay of perturbations of the uniform shearing solutions, we linearize 
the system  \eqref{eq:sys1U}-\eqref{eq:sys1G} around  \eqref{eq:rel_uniform}. 
Let $\delta \ll 1$ be a small parameter (describing the size of the perturbation) and consider  the asymptotic expansions
\begin{align*}
 U &= 1 + \delta \tilde U + O(\delta^2), \\
 \Gamma &= 1 + \delta \tilde{ \Gamma } + O(\delta^2), \\
 \Sigma &= 1 + \delta \tilde{ \Sigma } + O(\delta^2).
\end{align*}
By neglecting $O(\delta^2)$-terms, we obtain from \eqref{eq:sys1U}-\eqref{eq:sys1G} the linearized system
\begin{align}
 \tilde{U}_\tau &= (n\tilde{U} - \tilde{ \Gamma })_{xx}, \label{eq:linU}\\
 \tilde{ \Gamma }_\tau &= \tilde{U} - \tilde{ \Gamma }. \label{eq:linG}
\end{align}
The boundary condition \eqref{eq:bdry2} and the constraint  \eqref{constraintUG} imply
\begin{align}
  (n\tilde{U} - \tilde{ \Gamma })_{x}(0,t) = (n\tilde{U} - \tilde{ \Gamma })_{x}(1,t) = 0, \label{eq:pert_bdry}\\
 \int_0^1 \tilde \Gamma (x, t) \, dx = \int_0^1 \tilde{U}(x,t) \, dx = 0 \quad \text{for all $t>0$}.  \label{eq:pert_constraint} 
\end{align}

Growth and decay modes of the linearized system \eqref{eq:linU}-\eqref{eq:linG} can be captured via spectral analysis. 
Let us assume the even extension of $\tilde{ \Sigma } = n\tilde{U} - \tilde{ \Gamma }$ in $[-1,1]$, which is compatible with \eqref{eq:pert_bdry} and consider a cosine series expansion of $\tilde{ \Sigma }$,
\begin{align*}
 \tilde{\Sigma}(x,\tau) &= \hat{\Sigma}_0(\tau) + \sum_{j=1}^{\infty} \hat{\Sigma}_j(\tau) \cos (j\pi x).
\end{align*}
In view of \eqref{eq:linU},  
we have also
\begin{equation}
\label{modes}
\begin{aligned}
 \tilde{U}(x,\tau) &= \hat{U}_0(\tau) + \sum_{j=1}^{\infty} \hat{U}_j(\tau) \cos (j\pi x), \\
 \tilde{ \Gamma }(x,\tau) &= \hat{ \Gamma }_0(\tau) + \sum_{j=1}^{\infty} \hat{ \Gamma }_j(\tau) \cos (j\pi x).
\end{aligned}
\end{equation}
Then the  coefficients $( \hat{U}_j,  \hat{ \Gamma }_j)$ satisfy 
\begin{align}
 \begin{pmatrix}
  \hat{U}_j \\
  \hat{ \Gamma }_j
 \end{pmatrix}^\prime
 = 
 \begin{pmatrix}
  -nj^2\pi^2 & j^2\pi^2 \\
  1 & -1
 \end{pmatrix}
 \begin{pmatrix}
  \hat{U}_j \\
  \hat{ \Gamma }_j
 \end{pmatrix}, \qquad j \ge 0. \label{eq:matrix}
\end{align}
The characteristic polynomial of the coefficient matrix  is 
\begin{equation*}
 \lambda_j^2 + \lambda_j(1+n\pi^2j^2) - (1-n)\pi^2j^2 = 0,
\end{equation*}
with  discriminant 
$$
\Delta = (1+nj^2\pi^2)^2+4(1-n)j^2\pi^2 > 0, \quad 0<n<1.
$$
Thus two eigenvalues are real and 
$$
\lambda_{j,1} \lambda_{j,2} = - (1-n)\pi^2j^2 < 0.  
$$
Hence,  for $0\le n<1$, $j>0$ there is always one negative and one positive eigenvalue. The eigenvalues can be computed explicitly
\begin{align}
  \lambda^{\pm}_j &= \frac{1}{2}\Big(-(1+nj^2\pi^2)\pm\sqrt{(1+nj^2\pi^2)^2+4(1-n)j^2\pi^2}\Big). \label{eq:eigv}
\end{align}
The mode $j=0$ needs special attention. In this case the eigenvalues are $\lambda_{0,1}=0$ and $\lambda_{0,2}=-1$. Due to the constraint \eqref{eq:pert_constraint}, we have
\begin{align*}
& \hat{U}_0(\tau) \equiv 0, \\
& \hat{ \Gamma }_0^\prime(\tau) = - \hat{ \Gamma }_0(\tau),
\end{align*} 
thus the zero-th mode decays exponentially to zero.  We summarize the result:

\medskip
\noindent
{\bf Case 1.  $n=0$: Hadamard Instability}
\hfill\break
The eigenvalues are
  \begin{align}
    \lambda^{\pm}_j &= \frac{1}{2}\left(-1\pm \sqrt{1+4j^2\pi^2}\right) \, ,
    \quad \mbox{ $j \ge 0$},
     \label{eq:eigv_n0}
  \end{align}
  and satisfy $\lambda_j^+ > 0$ for $j >  0$ and $\lambda_j^- < 0$ for $j \ge 0$.

 Using the Taylor series expansion, $\sqrt{1+x} = 1 + \frac{1}{2}x - \frac{1}{4}x^2 + \frac{3}{8}x^3 + \cdots$, 
the leading order terms of \eqref{eq:eigv_n0}, are 
\begin{align*}
 \lambda^{\pm}_j &= \pm \pi j - \frac{1}{2} \pm \frac{1}{8\pi j} \mp \frac{1}{64\pi^3j^3} + O( \frac{1}{j^5}),
\end{align*}
so $\lambda^+_j$ increases linearly as $j \rightarrow \infty$ and is thus unbounded.
This kind of catastrophic instability, called \emph{Hadamard} instability,  is associated with the strain softening behavior and is 
typical in initial value problems of elliptic equations.

\medskip
\noindent
{\bf Case 2.  $n>0$: Turing Instability}  \hfill\break
Strain-rate dependence provides a diffusive mechanism, which moderates but does not entirely suppress the instability, 
as can be seen by the following lemma. The behavior in this regime is that of Turing instability 

\begin{lemma} \label{lem:turing}
For  $0<n<1$, the eigenvalues are
  \begin{align*}
    \lambda^+_j &= \frac{1}{2}\left(-(1+nj^2\pi^2)+\sqrt{(1+nj^2\pi^2)^2+4(1-n)j^2\pi^2}\right), \\
    \lambda^-_j &= \frac{1}{2}\left(-(1+nj^2\pi^2)-\sqrt{(1+nj^2\pi^2)^2+4(1-n)j^2\pi^2}\right).
  \end{align*}
and satisfy the properties
  \begin{enumerate}[(i)]
    \item $\lambda^+_j$ is increasing in $j$,
    \item $\lambda^+_j < \frac{1-n}{n}$, 
    \item $\lambda^+_j \rightarrow \frac{1-n}{n}$ \quad as\quad $j \rightarrow \infty$.
  \end{enumerate}
\end{lemma}

\begin{proof}
  Let $x=\pi j$ and write
  \begin{align*}
   \lambda^+_j(x) &= \frac{1}{2}(1+nx^2) \Bigg(-1+\sqrt{1 + \frac{4(1-n)x^2}{(1+nx^2)^2}}\Bigg) .
  \end{align*}  
  Since $\sqrt{1+z} < 1+ \frac{1}{2}z$ for $z>0$, we have 
  $$ \sqrt{1 + \frac{4(1-n)x^2}{(1+nx^2)^2}} < 1 + \frac{2(1-n)x^2}{(1+nx^2)^2},$$
  and hence
  \begin{equation} \label{eq:upperbound}
 \lambda^+_j(x) < \frac{(1-n)x^2}{1+nx^2} = \frac{1-n}{n + \frac{1}{x^2}} < \frac{1-n}{n} ,
  \end{equation}
  which proves $({\romannumeral 2})$.  The eigenvalue $\lambda^+_j(x)$ satisfies 
  $$ \lambda^+_j(x)^2 + \lambda^+_j(x)(1+nx^2) - (1-n)x^2 = 0.$$
  Differentiation with respect to $x$ gives
  \begin{gather*}
   (\lambda^+_j)^\prime(x) = 2x \frac{ (1-n) - n\lambda^+_j(x)}{2\lambda^+_j(x) + (1+nx^2)},
  \end{gather*}
  and the upper bound \eqref{eq:upperbound} of $\lambda^+_j(x)$ proves $({\romannumeral 1})$. 
  Further, the asymptotic expansion
  \begin{align*}
   \lambda^+_j(x)  
   &= \frac{1}{2} (1+nx^2)\bigg(\frac{1}{2}\frac{4(1-n)x^2}{(1+nx^2)^2} - \frac{1}{4}\Big(\frac{4(1-n)x^2}{(1+nx^2)^2}\Big)^2 + \cdots\bigg)\\
   &=\frac{1-n}{n + \frac{1}{x^2}} + O( \frac{1}{x^2} ) \quad \text{as $x \rightarrow \infty$},
  \end{align*}
  yields $({\romannumeral 3})$.
\end{proof}

{  
\begin{remark} \label{remarkTuring} \rm
(i)  In view of \eqref{rescale}, positive eigenvalues imply linearized instability and negative linearized stability. 
Invoking that $\tau = \log(1+ \frac{t}{\gamma_0})$, Lemma \ref{lem:turing} implies that the rates are of polynomial order and  the precise
rate of growth or decay can be easily computed.
The upper bound of growth rate indicates that the rate of growth is bounded and the bound is proportional to $\tfrac{1}{n}$.
By \eqref{modes}, the perturbations exhibit oscillatory response.

(ii)  The instability arising  in the  case $0 < n <1$ is a {\it Turing}-type instability, namely, the combined effect of two different stabilizing mechanisms 
leads to an instability: Note that the coefficient matrix of \eqref{eq:matrix} can be decomposed as  
\begin{align*}
  \begin{pmatrix}
  -n j^2\pi^2 & j^2\pi^2 \\
  1 & -1
 \end{pmatrix} =   
 \begin{pmatrix}
 -n j^2\pi^2 & j^2\pi^2 \\
  0 & 0
 \end{pmatrix} +
  \begin{pmatrix}
  0 & 0  \\
  1 & -1
 \end{pmatrix}.
\end{align*}
Taking note of \eqref{eq:pert_constraint} both matrices on  the right-hand-side give rise to marginally stable systems for the eigenmodes; 
however, the  matrix obtained as the sum of the two 
is equiped with a family of strictly positive eigenvalues. This decomposition corresponds to visualizing the linearized problem \eqref{eq:linU}-\eqref{eq:linG}
as the sum of two problems
\begin{equation}
\left \{ 
\begin{aligned}
 \tilde{U}_\tau &= (n\tilde{U} - \tilde{ \Gamma })_{xx},   \\
 \tilde{ \Gamma }_\tau &=  0  \end{aligned}
\right .
\quad \mbox{and} \quad 
\left \{
\begin{aligned}
 \tilde{U}_\tau &= 0 \, ,   \\
 \tilde{ \Gamma }_\tau &= \tilde{U} - \tilde{ \Gamma }. 
 \end{aligned}
 \right .
\end{equation}
which, subject to the restriction \eqref{eq:pert_constraint}, are both marginally stable. 
The stabilizing mechanism in the first one is viscosity, while in the second one are the inertial effects
(stemming from the fact that excess growth is needed to overcome the uniform shear solution). This perspective clarifies the stabilizing
mechanisms present in this problem; still, the prime driver of instability is the softening response of the system that effects the coupling of the two systems.

\end{remark}
}

\section{Derivation of an effective nonlinear equation} \label{sec:coarse_scale}

The analysis in Section \ref{sec:linear} concerns the behavior of the linearized problem and captures the onset of instability. 
Focusing next in the nonlinear regime, we devise an asymptotic criterion, in the
spirit of \cite{KT09},  for the onset of localization.
The goal is to derive an effective equation for the evolutions of $U$ and $\Gamma$ valid at a coarse space and time scale. To
this end, we consider the rescaling of independent variables
$$ s= \epsilon\tau, \qquad y=\sqrt{ \epsilon }\,x, $$
and rewrite \eqref{eq:sys1U}-\eqref{eq:sys1G} as
\begin{align}
 {U}_s &= \bigg( \frac{ {U}^n }{ \Gamma }\bigg)_{yy}, \label{eq:coarseU}\\
 \epsilon \Gamma_s &= U- \Gamma. \label{eq:coarseG}
\end{align}
For small values of $\epsilon$, we can view \eqref{eq:coarseU} as a moment equation and equation \eqref{eq:coarseG} as a relaxation process towards  the equilibrium curve 
$$
\Gamma=U, \quad \Sigma = U^{n-1}.
$$

We are interested in calculating the equation describing the effective response of \eqref{eq:coarseU}, \eqref{eq:coarseG} for $\epsilon$ sufficiently small. We consider a Chapman-Enskog type expansion with $\epsilon \ll 1$ for the field variables
\begin{align*}
 U &= U_0 + \epsilon U_1 + O(\epsilon^2), \\
 \Gamma &= \Gamma_0 + \epsilon \Gamma_1  + O(\epsilon^2).
\end{align*}
Then we have 
\begin{align*}
  U^n &= U_0^n\big(1 + \epsilon \frac{U_1}{U_0} + O( \epsilon^2 )\big)^n = U_0^n\bigg(1+ n \epsilon \frac{U_1}{U_0}\bigg) + O( \epsilon^2 ),\\
  \frac{1}{ \Gamma } &= \frac{1}{ \Gamma_0 } \frac{1}{1+ \epsilon \frac{ \Gamma_1 }{ \Gamma_0 } + O( \epsilon^2 )} = \frac{1}{ \Gamma_0 }\bigg(1- \epsilon\frac{ \Gamma_1 }{ \Gamma_0 }\bigg) + O( \epsilon^2 ),\\
  \frac{U^n}{ \Gamma } &= \frac{U_0^n}{ \Gamma_0 } + \epsilon \frac{U_0^n}{ \Gamma_0 }\bigg( n\frac{U_1}{U_0}  - \frac{ \Gamma_1} { \Gamma_0} \bigg) + O( \epsilon^2 ) \\
  &= \frac{U_0^n}{ \Gamma_0 }\bigg(1 + \epsilon \bigg( -(1-n)\frac{U_1}{U_0} + \frac{U_1}{U_0}  - \frac{ \Gamma_1} { \Gamma_0} \bigg)\bigg) + O( \epsilon^2 ).
\end{align*}
From \eqref{eq:coarseU} and \eqref{eq:coarseG}, collecting together the $O(1)$-terms, we have
\begin{subequations}
\begin{align}
 &U_0 - \Gamma_0 = 0, \\
 &\partial_s U_0 = \bigg(\frac{U_0^n}{ \Gamma_0 }\bigg)_{yy} = \big(U_0^{-(1-n)}\big)_{yy}. \label{eff0}
 \end{align}
\end{subequations}
This implies that the equation describing the effective dynamics at the order $O(\eps)$ is \eqref{eff0}. 
Since $0<n<1$, equation \eqref{eff0} is a backward parabolic equation. 
On the one hand, this indicates instability, on the other the asymptotic procedure will cease to be a good approximation at
this order $O(\eps)$.

We thus proceed to calculate the effective equation at the order $O(\eps^2)$.
Collecting together the $ \epsilon$-terms, we obtain
\begin{align*}
  \partial_s \Gamma_0 &= \partial_s U_0 = U_1 - \Gamma_1, \\
  \partial_s U_1 &= \bigg(\frac{U_0^n}{ \Gamma_0 }\bigg( -(1-n)\frac{U_1}{U_0} + \frac{U_1}{U_0}  - \frac{ \Gamma_1} { \Gamma_0} \bigg) \bigg)_{yy} \\
   &=\bigg(U_0^{-(1-n)}\bigg( -(1-n)\frac{U_1}{U_0} + \frac{1}{U_0} \Big(U_0^{-(1-n)}\Big)_{yy}\bigg) \bigg)_{yy}\\
   &=\bigg(U_0^{-(1-n)}\bigg( -(1-n)\frac{U_1}{U_0}\bigg)\bigg)_{yy} + \bigg(U_0^{-(2-n)}\bigg( \big(U_0^{-(1-n)}\big)_{yy}\bigg) \bigg)_{yy}.
\end{align*}
Hence,
\begin{align*}
 \partial_s U &= \partial_s U_0 + \epsilon\partial_s U_1 + O( \epsilon^2)\\
 &= \bigg(U_0^{-(1-n)}\bigg(1 - \epsilon(1-n)\frac{U_1}{U_0}\bigg)\bigg)_{yy} + \epsilon\bigg(U_0^{-(2-n)}\bigg( \big(U_0^{-(1-n)}\big)_{yy}\bigg) \bigg)_{yy} + O( \epsilon^2 )\\
 &=\bigg((U_0 + \epsilon U_1)^{-(1-n)}\bigg)_{yy} + \epsilon\bigg((U_0 + \epsilon U_1)^{-(2-n)}\bigg( \big((U_0 + \epsilon U_1)^{-(1-n)}\big)_{yy}\bigg) \bigg)_{yy} + O( \epsilon^2 )\\
 &=\bigg(U^{-(1-n)}\bigg)_{yy} + \epsilon\bigg(U^{-(2-n)}\bigg( \big(U^{-(1-n)}\big)_{yy}\bigg) \bigg)_{yy} + O( \epsilon^2 ).
\end{align*}
We thus obtain an effective equation for $U$  which captures the effective response of the system up to order $O(\epsilon^2)$,
\begin{equation}
 U_s = \bigg( \frac{1}{U^{1-n}} \bigg)_{yy} + \epsilon\bigg( \frac{1}{U^{2-n}} \bigg( \frac{1}{U^{1-n}} \bigg)_{yy} \bigg)_{yy}. \label{eq:effU}
\end{equation}

The leading term in \eqref{eq:effU} is backward parabolic while the first order correction is a fourth order.  The fourth order term introduces a  stabilizing mechanism to  the equation, in the sense that the linearized equation around the equilibrium $U = 1$ is stable. To see that,  write $U = 1 +  \tilde{U}$, 
and compute the linearized equation satisfied by  $ \tilde{U}$. This has the form
\begin{align}
\label{leff}
 \tilde{U}_s = -(1-n)\tilde{U}_{yy} - \epsilon\Big((2-n)\tilde{U}_{yy} + (1-n)\tilde{U}_{yyyy}\Big).  
\end{align}
To study the stability properties of \eqref{leff} we take the Fourier transform and obtain
 \begin{align*}
 \hat{U}_s = \Big(\big((1-n) + \epsilon(2-n)\big)\xi^2 - \epsilon(1-n)\xi^4\Big)\hat{U}.  
\end{align*}
We can see that the right-hand-side becomes negative with high frequency $\xi > \sqrt{ 1+ \frac{1}{ \epsilon} + \frac{1}{1-n}}$ due to the fourth order term. The fourth order term acts as a stabilizing mechanism.

\section{Emergence of Localization}\label{sec:localization}

Experimental observations of shear bands indicate that the development of shear localization proceeds in a fast but organized and coherent fashion.
It is not associated with oscillatory response, rather a coherent structure emerges from the process leading to strain localization. The reader should note
that the unstable modes of the linearized problem \eqref{eq:linU}-\eqref{eq:linG} for $0< n < 1$ contain oscillatory modes; such individual oscillatory modes are not observed in
experiments. The conjecture is that the combined effect of instability and nonlinearity suppresses the oscillations and results in a ``runaway" of the strain
at a point. In this section, we will construct a class of self-similar solutions that captures this process for the system \eqref{eq:mombal2}-\eqref{eq:kinematic2}.

The goal is to exploit the invariance properties of the system \eqref{eq:sys1U}-\eqref{eq:sys1G} in order to construct self-similar focusing solutions
that depict the initial stage of localization. We focus on values of the parameter space $0<n<1$ and carry out the following steps : 

\begin{itemize}
\item[(i)] In Section \ref{sec:emergence}, we consider  a variant of \eqref{eq:sys1U}-\eqref{eq:sys1G} and show that  its invariance 
properties suggest a similarity class of solutions of the form \eqref{eq:selfsimilar} with the profile $\big(\bar{V}, \bar{ \Gamma }, \bar{ \Sigma }\big)$
in  \eqref{eq:selfsimilar} satisfying \eqref{eq:barsysv}-\eqref{eq:barsyss}.

\item[(ii)] In Section \ref{sec:assds}, the system \eqref{eq:barsysv}-\eqref{eq:barsyss} is de-singularized and  transformed 
 to a system of three autonomous differential equations \eqref{eq:p}-\eqref{eq:r}. The profile 
$\big(\bar{V}, \bar{ \Gamma }, \bar{ \Sigma }\big)$ is transformed to a heteroclinic connection for the orbit $(p,q,r)$ of \eqref{eq:p}-\eqref{eq:r}.

\item[(iii)] In Section \ref{sec:equilibria}, we identify the equilibria $M_i$, $i =0, 1, 2, 3$, of the system \eqref{eq:p}-\eqref{eq:r}
and calculate their corresponding stable and unstable manifolds $W^u(M_i)$ and $W^s(M_i)$.

\item[(iv)] In Section \ref{sec:hetero}, we probe the possibilities for heteroclinic connections between the equilibria, and  
characterize our target orbit  as a connection  joining $M_0$ and $M_1$ (following  $\vec{X}_{02}$ near $M_0$). 

\item[(v)]  The  heteroclinic connection is computed numerically  in Section \ref{sec:numerical_const} as an orbit between $M_0$ and $M_1$.  
The implications on localizing solutions are summarized in Section \ref{sec:localiz}.

\end{itemize}


\subsection{Scale invariance and a class of similarity solutions} \label{sec:emergence}
We return to the system \eqref{eq:mombal}-\eqref{eq:kinematic}, restrict henceforth to $0<n<1$, 
and consider it in the entire domain $(x,t) \in \mathbb{R}\times [0,+\infty)$,
with the objective to construct a special class of solutions. Upon introducing the change of variables \eqref{rescale}, we obtain 
\begin{align}
 V_\tau &= \Sigma_{x}, \label{eq:sysV}\\ 
 \Gamma_\tau &= V_x- \Gamma, \label{eq:sysG}\\
 \Sigma &= \frac{(V_x)^n}{\Gamma}, \label{eq:sysS} 
\end{align}
where $V(x, \tau(t)) = v(x,t)$, and the last system is considered for $(x,\tau) \in \mathbb{R}\times [0,+\infty)$. 
Equation \eqref{eq:sysS} may be expressed as
\begin{equation*}
 \begin{split}
  V_x &= \big(\Sigma \Gamma)^{ \frac{1}{n} }, \quad \text{if $n>0$}, \\
  \Sigma &= \frac{1}{ \Gamma }, \quad \text{if $n=0$}.
 \end{split}
\end{equation*}
The system \eqref{eq:sysV}-\eqref{eq:sysS} is scale invariant. Indeed, given  $\big(V(x,\tau), \Gamma(x,\tau),\Sigma(x,\tau)\big)$, one verifies that
$\big(V_A(x,\tau), \Gamma_A(x,\tau), \Sigma_A(x,\tau)\big)$ defined by
\begin{equation} \label{eq:inv}
\begin{split}
  V_A(x,\tau) &= A^{ \frac{n}{2-n}} \,V( A x, \tau), \\
  \Gamma_A(x,\tau) &= A^{ \frac{2}{2-n}} \,\Gamma( A x, \tau),\\
 \Sigma_A(x,\tau) &= A^{ -1+\frac{n}{2-n}} \,\Sigma( A x, \tau),
\end{split}
\end{equation}
satisfies again \eqref{eq:sysV}-\eqref{eq:sysS}.
Note that the time $\tau$ is not rescaled here. 
The scaling invariance property motivates to seek for self-similar solutions of \eqref{eq:sysV}-\eqref{eq:sysS} in the form
\begin{equation} \label{eq:selfsimilar}
\begin{split}
  V(x,\tau) &= \phi_{\lambda}(\tau)^{ \frac{n}{2-n}} \,\bar{V}\big( \xi\big), \\
  \Gamma(x,\tau) &= \phi_{\lambda}(\tau)^{ \frac{2}{2-n}} \,\bar\Gamma\big( \xi\big),\\
 \Sigma(x,\tau) &= \phi_{\lambda}(\tau)^{ -1+\frac{n}{2-n}} \,\bar\Sigma\big( \xi\big),
\end{split}
\end{equation}
functions of the similarity variable
\begin{equation*}
 \xi = x\,\phi_{\lambda}(\tau),
\end{equation*}
where $\phi_{\lambda}(\tau)$ is specified later. (The format of \eqref{eq:selfsimilar} is suggested by setting $A \to \phi_\lambda (\tau)$ in \eqref{eq:inv}
thus exploiting the property that rescalings of $\tau$ do not enter in the scale invariance property \eqref{eq:inv}).

With this form of solutions, the problem of localization is transformed to the problem of construction of an appropriate self-similar ``profile" functions
$\big(\bar{V}, \bar{ \Gamma }, \bar{ \Sigma }\big)$. The reason is the following: Suppose a smooth profile $\big(\bar{V}, \bar{ \Gamma }, \bar{ \Sigma }\big)$ 
is constructed associated to an increasing function $\phi_{\lambda}(\tau)$ satisfying $\phi_{\lambda}(\tau) \to \infty$ as $\tau \to \infty$.
Then \eqref{eq:selfsimilar} will describe a coherent localizing structure where  $\big(\bar{V}, \bar{ \Gamma }, \bar{ \Sigma }\big)$ will depict the profile
of that structure. Indeed, if for example $\bar\Gamma$ is a bell-shaped even profile, the transformation of variables  will have the effect of narrowing the width 
and increasing the height of the profile as time increases. Eventually it forms a singularity at $x=0$ as $\tau \to +\infty$.

A simple calculation shows that the functions  $\big(\bar{V}, \bar{ \Gamma }, \bar{ \Sigma }\big)$ satisfy a system of ordinary differential equations,
\begin{align}
 \lambda \left( \frac{n}{2-n} \bar{V} + \xi \bar{V}_\xi\right) &=\bar{\Sigma}_\xi,  \label{eq:barsysv}\\
 \lambda \left( \frac{2}{2-n} \bar{ \Gamma} + \xi \bar{ \Gamma}_\xi\right) &= \bar{V}_\xi - \bar{ \Gamma }, \label{eq:barsysg} \\
 \bar{ \Sigma } &= \frac{\bar{V}_\xi^n}{\bar{\Gamma}}, \label{eq:barsyss} 
\end{align}
where we have selected $ \lambda:=\frac{\dot{\phi}_\lambda (\tau)}{\phi_\lambda (\tau) }$ independently of $\tau$, or equivalently
\begin{equation*}
  \phi_{\lambda}(\tau)=e^{\lambda \tau}, ~~ \text{with \; $\phi_\lambda (0)=1$}.
\end{equation*}
In this ansatz the parameter  $\lambda$ parametrizes the focusing rates of growth and narrowing during the  developing  localizing structure. 
Invoking $\tau = \log(1+ \frac{t}{\gamma_0})$, we see that $\phi_{ \lambda }(\tau(t)) = (1+ \frac{t}{\gamma_0})^\lambda$ and thus we are seeking a family of solutions
growing at a polynomial order.

The system \eqref{eq:barsysv}-\eqref{eq:barsyss} is solved subject to initial data
\begin{align}
\Gamma (0)  &= \bar \Gamma_0
\label{eq:barcond2}
\\
\bar U (0) = \bar V_\xi (0) &= \bar U_0 \, ,   \quad \bar \Sigma (0) = \bar \Sigma_0 \, .
\label{barid}
\end{align} 
Because the problem is singular the values of the data cannot be independent and \eqref{eq:barsysg}-\eqref{eq:barsyss} imply
the compatibility conditions
\begin{align}
 \bar \Sigma_0 = \frac{ {\bar U_0}^n }{ \bar \Gamma_0},
 \label{idcompat1}
 \\
\left ( 1 +  \lambda \tfrac{2}{2-n}  \right ) \bar \Gamma_0  = \bar U_0 \, .
\label{idcompat2}
\end{align}
 Therefore, among the five parameters $n$, $\lambda$, $\bar \Gamma_0$, $\bar U_0$ and $\bar \Sigma_0$, involved in determining these solutions,
 only three are independent. We adopt the view that $n$ and the rate $\lambda$ precise the system \eqref{eq:barsysv}-\eqref{eq:barsyss}, while
 one parameter (say) $\bar \Gamma_0$ reflects choice of the initial data.

Further inspection of \eqref{eq:barsysv}-\eqref{eq:barsyss} shows that if $(\bar{V}(\xi), ~\bar{\Gamma}(\xi),~ \bar{\Sigma}(\xi))$ is a solution, then $(-\bar{V}(-\xi), ~\bar{\Gamma}(-\xi),~ \bar{\Sigma}(-\xi))$  is also a solution. This symmetry implies $\bar{V}(\xi)$  is odd,  $\bar{ \Gamma}(\xi)$ and $\bar\Sigma(\xi)$  are even, and suggests to impose
\begin{align}
 \bar{V}(0) = 0, \qquad \bar\Gamma_\xi(0) = 0, \qquad \bar\Sigma_\xi(0) = 0. \label{eq:barcond1}
 \end{align}
In summary, we seek a family of solutions to \eqref{eq:barsysv}-\eqref{eq:barsyss} in the half space $\xi \in \mathbb{R}^+$ 
subject to the initial conditions \eqref{eq:barcond2} and \eqref{eq:barcond1}.


The scale invariance property of \eqref{eq:sysV}-\eqref{eq:sysS} is inherited by the system \eqref{eq:barsysv}-\eqref{eq:barsyss}:
\begin{equation} \label{eq:inv_bar}
\begin{split}
  \bar{V}_A(\xi) &= A^{\frac{n}{2-n}}\,\bar{V}\big( A{\xi} \big),  \\
  \bar{ \Gamma }_A(\xi) &= A^{\frac{2}{2-n}}\,\bar{\Gamma}\big( A{\xi} \big),\\
 \bar{ \Sigma }_A(\xi) &= A^{ -1+\frac{n}{2-n}}\,\bar{\Sigma}\big( A{\xi} \big), 
\end{split}
\end{equation}
is a solution if $\big(\bar{V}(\xi)$, $\bar{\Gamma}(\xi), \bar{ \Sigma }(\xi)\big)$ is a solution. Note that conditions \eqref{eq:barcond1} persist under rescaling
but the initial value \eqref{eq:barcond2} does not. 
In other words, once we have a solution with the initial value $\bar\Gamma_0=1$, then we obtain solutions with different initial values by rescaling.

\subsection{An associated dynamical system}
\label{sec:assds}
In the previous section the emergence of a localizing structure was reduced to the existence of an orbit for the initial value problem
 \eqref{eq:barsysv} and \eqref{eq:barsyss} satisfying the initial conditions \eqref{eq:barcond2}-\eqref{eq:barcond1}. 
The system  \eqref{eq:barsysv}-\eqref{eq:barsyss} is non-autonomous and singular at $\xi=0$. 
In this section, we introduce a series of transformations in order to de-singularize the problem and 
obtain an autonomous system of ordinary differential equations. The solution associated with a localizing structure 
rises as a heteroclinic orbit of the new system.

Define $\big(\tilde{v}(\eta)$, $\tilde{ \gamma}(\eta)$, $\tilde{ \sigma}(\eta)$, $\tilde{ u}(\eta)\big)$ and a new independent variable $\eta$ by
\begin{equation} \label{eq:transform}
\begin{aligned}
  \tilde{v}\big(\log\xi\big) &=\xi^{ \frac{n}{2-n}} \,\bar{V}(\xi), &  
  \tilde{ \gamma}\big(\log\xi\big) &=\xi^{ \frac{2}{2-n}} \,\bar{ \Gamma}(\xi), \\
  \tilde{ \sigma}\big(\log\xi\big) &=\xi^{ -1+\frac{n}{2-n}} \,\bar{ \Sigma}(\xi), &
  \tilde{u}\big(\log\xi\big) & =\xi^{ \frac{2}{2-n}} \bar{U}(\xi), 
\end{aligned}
\end{equation}
where
 $\eta = \log\xi,  \ \eta \in [-\infty, +\infty] \, $
and $\bar{U}(\xi)$ is such that
$$U (x,\tau) = \phi_{\lambda}(\tau)^{ \frac{2}{2-n}}\bar{U}(\xi).$$ 
Then $\big(\tilde{v}(\eta)$, $\tilde{ \gamma}(\eta)$, $\tilde{ \sigma}(\eta)\big)$ satisfy the autonomous system 
\begin{align*}
  \tilde{\sigma}^\prime &=-\Big(1- \frac{n}{2-n}\Big) \tilde\sigma + \frac{ \lambda n}{2-n}\tilde{v} + \lambda(\tilde\sigma \tilde\gamma)^{ \frac{1}{n} },\nonumber\\
  \tilde{\gamma}^\prime &= \frac{1}{ \lambda } \big((\tilde\sigma \tilde\gamma)^{ \frac{1}{n} }- \tilde\gamma\big), \nonumber\\
  \tilde{v}^\prime &= (\tilde\sigma \tilde\gamma)^{ \frac{1}{n} } + \frac{n}{2-n} \tilde{v}, 
\end{align*}
in the new independent variable $\eta$, where we have denoted $\frac{d}{d\eta}(\cdot)$ by $(\cdot)^\prime$.

We observe that even if $ \tilde{ \gamma} \rightarrow \tilde{ \gamma }_\infty$ as $\eta \rightarrow \infty$ for some $\tilde{ \gamma }_\infty<\infty$, then $(\tilde\sigma \tilde\gamma)^{ \frac{1}{n} } \rightarrow \tilde{ \gamma }_\infty$ and thus $\tilde{v}$ and $ \tilde{ \sigma }$ diverge to infinity. 
This causes analytical difficulties, and in order to overcome this difficulty we introduce a second non-linear transformation 
to obtain an equivalent system with all equilibria lying in a finite region. We define new variables
\begin{align} \label{eq:defpqr}
 p = \frac{ \tilde\gamma}{ \tilde\sigma}, \qquad  q = n \frac{\tilde v}{ \tilde\sigma}, \qquad 
 r = \big( \tilde\sigma \tilde\gamma^{1-n}\big)^{\frac{1}{n}} ~~\Big(= \frac{\tilde u}{ \tilde\gamma }\Big).
\end{align}
A cumbersome but straightforward computation shows that $(p(\eta), q(\eta), r(\eta))$ satisfies the autonomous system
\begin{align}
  p^\prime &=p\Big( \frac{1}{\lambda}(r-1) + \big(1-\frac{n}{2-n}\big) -\lambda pr - \frac{\lambda}{2-n}q\Big), \label{eq:p}\\
  q^\prime &=q\Big(1-\lambda pr - \frac{\lambda}{2-n}q\Big) + npr, \label{eq:q}\\
  nr^\prime &=r\Big(\frac{1-n}{\lambda}(r-1) +\big(-1+\frac{n}{2-n}\big) + \lambda pr + \frac{\lambda}{2-n}q\Big). \label{eq:r}
\end{align}
In the following section we study the equilibria of system \eqref{eq:p}-\eqref{eq:r} and the local behavior around them,
with the objective of understanding the existence of heteroclinic orbits between the equilibria.


\subsection{Equilibria and orbits} \label{sec:equilibria}
System \eqref{eq:p}-\eqref{eq:r} has four equilibria, for $ \lambda \ne 1 + \frac{n}{2(1-n)}$, which are
\begin{equation}
\label{equilibria}
\begin{aligned}
 M_0 &= \begin{pmatrix}
        p_0 \\ q_0 \\ r_0
       \end{pmatrix}
     = \begin{pmatrix} 
        0 \\ 0 \\ 1+ \frac{2\lambda}{2-n}
       \end{pmatrix}, 
\qquad 
&& M_1 = \begin{pmatrix}
        p_1 \\ q_1 \\ r_1
       \end{pmatrix}
     = \begin{pmatrix} 
        0 \\ \frac{2-n}{ \lambda } \\ 1- \frac{ n\lambda}{(2-n)(1-n)}
       \end{pmatrix}, \\     
 M_2 &= \begin{pmatrix}
        p_2 \\ q_2 \\ r_2
       \end{pmatrix}
     = \begin{pmatrix} 
        0 \\ \frac{2-n}{ \lambda } \\ 0
       \end{pmatrix}, 
 \qquad 
 &&M_3 = \begin{pmatrix}
        p_3 \\ q_3 \\ r_3
       \end{pmatrix}
     = \begin{pmatrix} 
        0 \\ 0 \\ 0
       \end{pmatrix}.   
\end{aligned}
\end{equation}
When $ \lambda = 1 + \frac{n}{2(1-n)}$ the equilibria $M_0$, $M_1$, $M_2$ persist, but $M_3$ is replaced by equilibria
distributed on the entire $p$-axis. Since, we will be interested in the heteroclinic from $M_0$ to $M_1$ we do not consider this case separately.
Furthermore, we restrict attention to the placement that  $M_1$ lies above the plane $r=0$, which implies the restriction
\begin{equation}
 \lambda < \frac{(2-n)(1-n)}{n}. \label{eq:lambda_upperbdd}
\end{equation}
The placement of equilibria is depicted in Figure \ref{fig:equilibria0}.

The local behavior of the dynamical system around each equilibrium is determined (for hyperbolic equilibria) from the eigenstructure of
the linearized problem. We denote the three eigenvalues of the linearization at each $M_i$ by $\mu_{ij}$ and the associated eigenvectors 
by $\vec{X}_{ij}$, $j=1,2,3$.

$\bullet$ $M_0$ \emph{is as unstable node}: all eigenvalues are real and positive. The first and the last eigenvectors lie on the $(q,r)$-plane, 
the first one is pointing towards the origin on the $r$-axis and the last one is pointing towards right and down on the $(q,r)$-plane. 
The second eigenvector is going off the $(q,r)$-plane. (See Figure \ref{fig:equilibria0}.) 
\begin{align}
  \vec{X}_{01} &= \Bigg(0 , 1 , - \frac{\lambda}{1-n} \frac{\lambda}{2-n} \bigg(\frac{1}{1- \frac{n}{1-n} \frac{\lambda}{r_0}} \bigg)\Bigg), & & \mu_{01}=1,
  \nonumber
  \\
  \vec{X}_{02} &= \Bigg( \frac{1}{r_0} , n , - \frac{\lambda}{1-n} \frac{\lambda}{2-n} \bigg(\frac{1}{ \frac{1}{2}- \frac{n}{1-n} \frac{\lambda}{r_0}} \bigg)\Bigg), & & \mu_{02}=2,
  \label{equivector}
  \\
  \vec{X}_{03} &= (0 , 0 , 1), & & \mu_{03}= \frac{1}{n} \frac{1-n}{ \lambda } r_0.
  \nonumber
\end{align}

$\bullet$ $M_1$ \emph{is a saddle }: all eigenvalues are real, two are negative and one is positive. The first and the last eigenvector lie on the $(q,r)$-plane, 
the first one is an unstable direction, pointing toward the $M_2$ and the last one is a stable direction, pointing toward $M_1$ in right bottom direction. The second eigenvector is pointing off the plane. 
\begin{align*} 
  \vec{X}_{11} &= \Bigg( 0, 1,   - \frac{\lambda}{1-n} \frac{\lambda}{2-n} \bigg(\frac{1}{1+ \frac{n}{1-n} \frac{ \lambda}{r_1}} \bigg)\Bigg), \quad\mu_{11}=-1,\\
  \vec{X}_{12} &= \Bigg( -\frac{1}{r_1}\Big(2- \frac{1}{1-n}\Big) \Big(\frac{2-n}{ \lambda }\Big),2(1-n),  - \left( \frac{n}{1-n}\right)\frac{\lambda}{1-n} \frac{\lambda}{2-n} \bigg(\frac{1}{1+ \big(\frac{n}{1-n}\big)^2 \frac{\lambda}{r_1}} \bigg)\Bigg),\\
  &\quad\quad\quad \mu_{12}=- \frac{n}{1-n},\\
  \vec{X}_{13} &= (0,  0, 1), \quad \mu_{13}= \frac{1-n}{ \lambda n } r_1.
\end{align*}

$\bullet$ $M_2$ \emph{is a stable node}: all eigenvalues are real and negative. The eigenvectors are the coordinate basis vectors,
\begin{align*}
  \vec{X}_{21} &= (1,0,0), \quad \mu_{21}=-\frac{1}{ \lambda }\left(1+ \frac{n}{ 2-n } \lambda\right),\\
  \vec{X}_{22} &= (0,1,0), \quad \mu_{22}=-1,\\
  \vec{X}_{23} &= (0,0,1), \quad \mu_{23}=-\frac{1}{n} \frac{1-n}{ \lambda } r_1.
\end{align*}

$\bullet$ $M_3$ \emph{is at the origin and is a saddle}: all eigenvalues are real, the second eigenvalue is positive and the last eigenvalue is negative. The sign of the first eigenvalue bifurcates by $ \lambda =  1+ \frac{n}{2(1-n)}$; it is negative if $ \lambda < 1+ \frac{n}{2(1-n)}$ and is positive if $ \lambda > 1+ \frac{n}{2(1-n)}$. The eigenvectors are the coordinate basis vectors,
\begin{align*}
  \vec{X}_{31} &= (1,0,0), \quad \mu_{31}=- \frac{1}{ \lambda }\left(1- \frac{2(1-n)}{2-n} \lambda\right),\\
  \vec{X}_{32} &= (0,1,0), \quad \mu_{32}=1,\\
  \vec{X}_{33} &= (0,0,1), \quad \mu_{33}=-\frac{1}{n} \frac{1-n}{ \lambda } r_0.
\end{align*}

\subsection{Characterization of the heteroclinic orbit} \label{sec:hetero}

In this section, we study the heteroclinic connections between the equilibrium points. Since there are four equilibria, the unstable and stable manifolds of the equilibria 
might intersect in various ways and could conceivably produce multiple connections. We aim to identify the one that is physically relevant and to construct it in phase space. 
We assert that the relevant orbit  joins  $M_0$ to $M_1$ as $\eta$ runs from $-\infty$ to $\infty$. 
 A schematic sketch of this orbit is depicted in Figure \ref{fig:hetero}. Let us explain how that is  singled out.
\begin{figure}
  \centering
  
  \subfigure[Equilibria of $(p,q,r)$-system and linearized vector fields around]{
  \includegraphics[width=7cm]{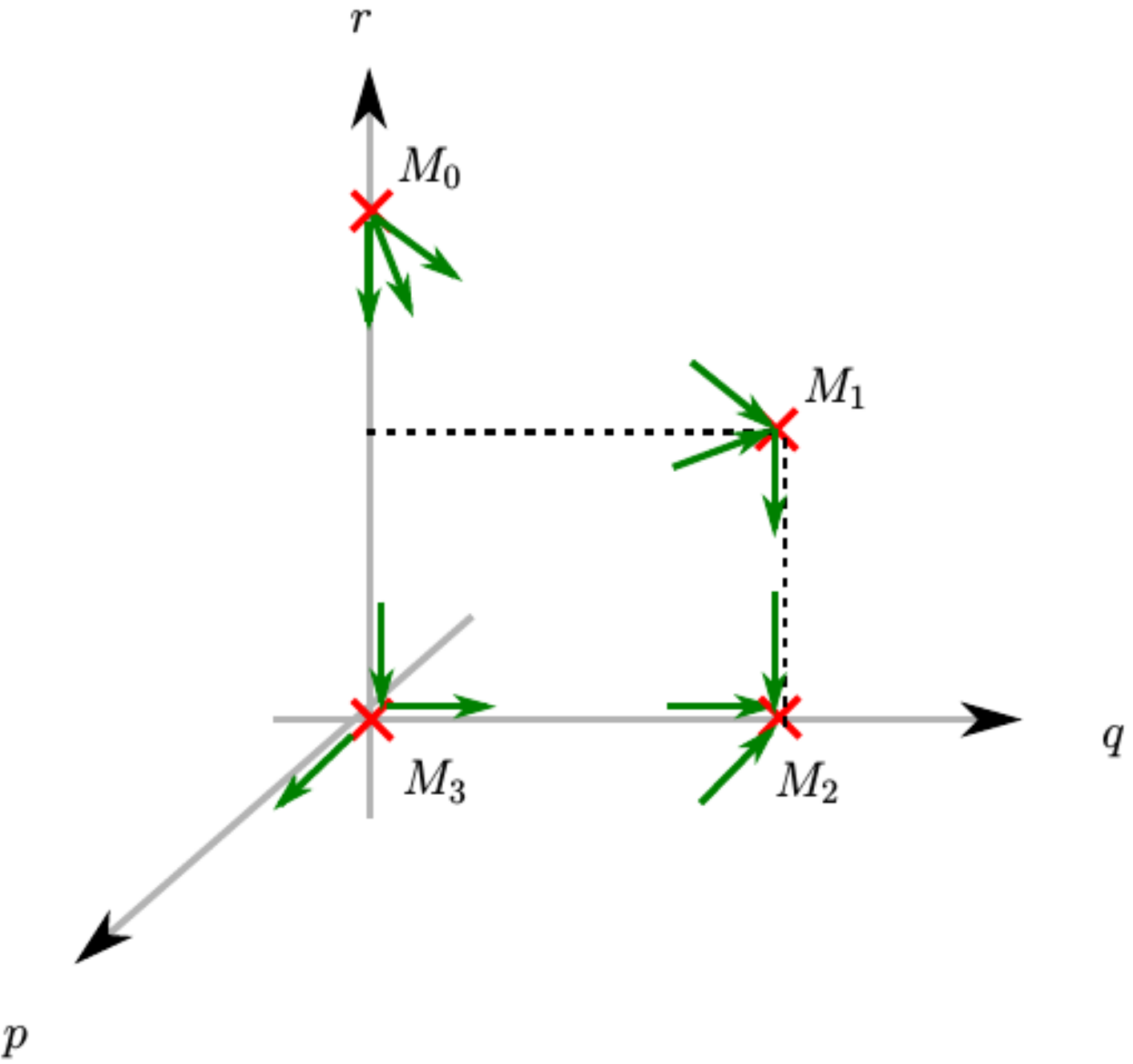} \label{fig:equilibria0}
   } 
  \subfigure[Schematically illustrated target orbit between $M_0$ and $M_1$] {
   \includegraphics[width=7cm]{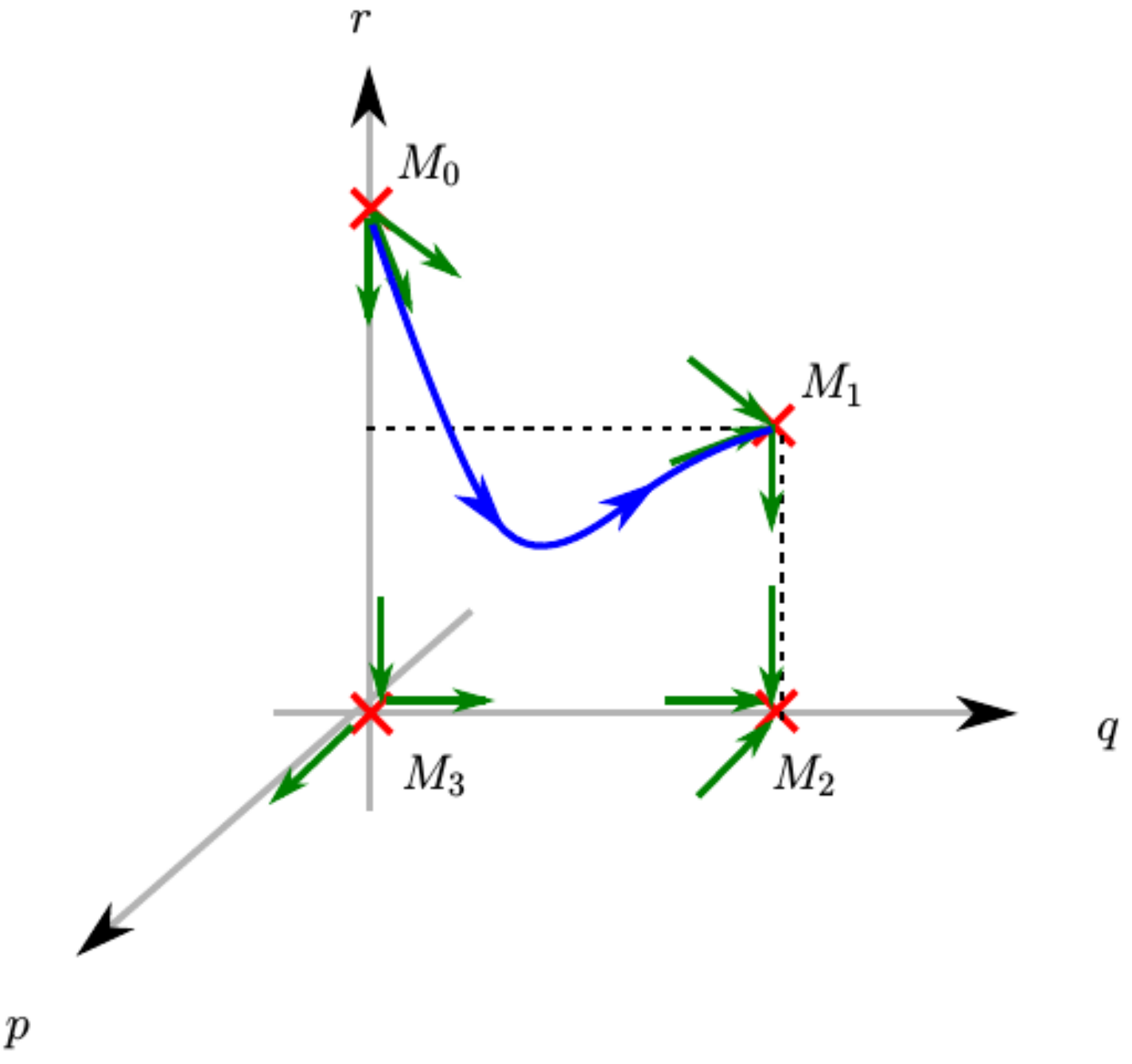} \label{fig:hetero}
  }
  \caption{Phase diagram. Features attributed to generic parameters are  $(i)$ $M_1$ lies above the plane $r=0$; $(ii)$ $\mu_{31}>0$. }
\end{figure}

\subsubsection{The behavior as $\eta \rightarrow \infty$} \label{sec:char_infty}

We will be interested in the orbit converging to the equilibrium $M_1$. The reason is the following: From the perspective of shear band formation, the strain of a 
localizing solution should grow with time, as the material is loading. Note that if $\gamma(x,t)$ grows polynomially as $t^\rho$, then the ratio
$ \frac{(t  + \gamma_0)\gamma_t}{\gamma} \sim \rho$ as $t$ increases. On the other hand, the transformations we imposed imply
\begin{align*}
 r=\frac{\tu}{\tg} = \frac{ \xi^{ \frac{2}{2-n} } \, \bU(\xi) }{ \xi^{ \frac{2}{2-n} } \, \bG(\xi) } = \frac{\bU(\xi)}{\bG(\xi)} = \frac{U(x,\tau)}{\Gamma(x,\tau)}
 =\frac{(t+ \gamma_0)\bar{u}(x,t)}{\bar\gamma(x,t)} = \frac{(t  + \gamma_0)\bar\gamma_t}{\bar\gamma}.
\end{align*}
Thus we are interested in  $r \to \rho$ with $\rho > 0$ as $\eta \to \infty$, which suggests to restrict attention to orbits converging to $M_1$.

The behavior of the nonlinear problem near the saddle point $M_1$ is determined by the corresponding linearization and the 
orbit $\varphi(\eta)$ in a neighbourhood of $M_1$ is expressed as
 \begin{equation}
   \varphi(\eta)  - M_1
  = \kappa^\prime_1 e^{- \eta}\vec{X}_{11} + \kappa^\prime_2 e^{- \frac{n}{1-n}\eta}\vec{X}_{12} + \textrm{higher-order terms} \quad \text{as $\eta \rightarrow \infty$} \label{eq:local1} \, .
\end{equation}
We are interested in orbits that have a nontrivial component out of the plane $p = 0$. This is achieved by requiring the coefficient $\kappa_2' \ne 0$ in \eqref{eq:local1}. 
Note that when $\frac{n}{1-n} <1$ the term $\kappa^\prime_2 e^{- \frac{n}{1-n}\eta}\vec{X}_{12}$ determines the asymptotic response when $\kappa_2' \ne 0$.

The reader should note that the plane $p=0$ is invariant for the flow of the dynamical system \eqref{eq:p} - \eqref{eq:r}. In fact the flow can be explicitly
computed in this case and provides the heteroclinic connecting $M_0$ to $M_1$ on the plane $p =0$. As this orbit is known we focus on the ones that venture
out of the plane $p=0$, hence imposing $\kappa_2' \ne 0$.  Figure \ref{fig:hetero} illustrates schematically   the orbit approaching $M_1$ in $\vec{X}_{12}$.

\subsubsection{The behavior as $\eta \rightarrow -\infty$}
The orbit emanating from $M_0$  (as $\eta \to -\infty$) has the asymptotic expansion
\begin{equation} \label{eq:asymp_alpha}
   e^{-2\eta}\left[\begin{pmatrix}
   p(\eta) \\ q(\eta) \\ r(\eta)
  \end{pmatrix}
  - M_0
  \right] \rightarrow \kappa\vec{X}_{02}, \quad \text{as $\eta \rightarrow -\infty$ for some constant $\kappa>0$.}
\end{equation}
This claim  follows from the initial conditions \eqref{eq:barcond2} and \eqref{eq:barcond1}, via an asymptotic analysis  of the behavior of  the solution $(\bar{V}(\xi), ~\bar{\Gamma}(\xi),~ \bar{\Sigma}(\xi))$ of 
\eqref{eq:barsysv}-\eqref{eq:barsyss} near the singular point $\xi = 0$, which is presented in the Appendix.
Figure \ref{fig:hetero} depicts the orbit exhibiting the above asymptotics; the orbit emerges in the direction of $\vec{X}_{02}$ of the second unstable eigenspace.

\subsubsection{Singling out the heteroclinic orbit}
\label{sec:singling}
The numerical calculations performed in Section \ref{sec:numerical_const}  (see Fig \ref{fig:2manifold}) suggest there is a
two-parameter family of heteroclinic orbits connecting $M_0$ and $M_1$. This conjecture is also supported by the following arguments:
$M_1$ has a two dimensional stable manifold while  $M_0$ is an unstable node (with a three-dimensional unstable manifold).
One heteroclinic connection between $M_0$ and $M_1$ lies in the plane $p =0$ and can be  explicitly computed. Its computation results
 from noticing that the plane $p=0$ is invariant under the flow of the dynamical system \eqref{eq:p}-\eqref{eq:r} and integrating the resulting system of (two) differential 
 equations. This provides one connection on the plane $p=0$. Since, the numerical computations suggest there is
 an out-of-plane heteroclinic connection this indicates that there is a two parameter family of such connections. Indeed, they are systematically
 computed in Section \ref{sec:numerical_const}.

Given the two parameter family, we proceed to show how to select the desired heteroclinic that satisfies \eqref{eq:barcond2} and \eqref{eq:barcond1}.
The objective is to select the heteroclinic connection that satisfies the asymptotic relation  \eqref{eq:asymp_alpha}. 
Note that using \eqref{idcompat1} and \eqref{idcompat2} we can  compute from the data  the coefficient  $\kappa = \bG_0\bU_0^{1-n}$ in \eqref{eq:asymp_alpha}.

The general heteroclinic orbit $\varphi(\eta)$ satisfies,  in a sufficiently small neighbourhood of $M_0$, the expansion
\begin{equation}
   \varphi(\eta)
  - M_0
  = \kappa_1 e^{\eta}\vec{X}_{01} + \kappa_2 e^{2\eta}\vec{X}_{02} + \kappa_3 e^{\frac{1}{n} \frac{1-n}{ \lambda } r_0\eta}\vec{X}_{03}+ \text{higher-order terms}, \label{eq:local}
\end{equation}
as $\eta \rightarrow -\infty$. For $n$ small, $1<2<\frac{1}{n} \frac{1-n}{ \lambda } r_0$, the first term in the right hand side  of \eqref{eq:local} dominates the remaining terms. 
Therefore,  \eqref{eq:asymp_alpha} dictates  that  the coefficient $\kappa_1=0$ and thus
\eqref{eq:asymp_alpha} fixes one curve and reduces  by one the degrees of freedom.
Let us denote by 
$\Phi(\eta) := \left( P(\eta), \ Q(\eta), \ R(\eta) \right)$
a heteroclinic orbit that emanates  in
the direction of the eigenvector $\vec{X}_{02}$. There is one degree of freedom at our disposal
provided by the translation invariance of the selected orbit.

Having selected the curve $\Phi (\eta)$, we proceed to select the translation factor $\eta_0$ as follows:
Let $\kappa_2$ be the  constant associated to the asymptotic behavior of $\Phi(\eta)$ in  \eqref{eq:local}. We set our target heteroclinic
$$
\varphi^\star (\eta) = \begin{pmatrix}
   p(\eta) \\ q(\eta) \\ r(\eta)
  \end{pmatrix}=\begin{pmatrix}
   P(\eta+\eta_0) \\ Q(\eta+\eta_0) \\ R(\eta+\eta_0)
  \end{pmatrix},
$$
and compute

\begin{align*}
&e^{-2\eta}\left[\begin{pmatrix}
   p(\eta) \\ q(\eta) \\ r(\eta)
  \end{pmatrix}
  - M_0\right] = e^{-2\eta}\left[\begin{pmatrix}
   P(\eta+\eta_0) \\ Q(\eta+\eta_0) \\ R(\eta+\eta_0)
  \end{pmatrix}
  - M_0\right]\rightarrow {\kappa_2}e^{2\eta_0}\vec{X}_{02}, \quad \text{as $\eta \rightarrow -\infty$.}
\end{align*}
Therefore, given $\kappa$ in \eqref{eq:asymp_alpha}, we compute $\eta_0$ by 
\begin{equation}
 \eta_0 = \frac{1}{2}\log\left(\frac{\kappa_2}{\kappa}\right). \label{eq:eta0}
\end{equation}
In summary,  there exists a two parameter family of heteroclinic orbits joining $M_0$ and $M_1$,  whose existence is supported by the numerical calculations, 
and \eqref{eq:asymp_alpha} singles out one orbit among them.

\subsection{A three-parameter family of focusing self-similar solutions} \label{sec:localiz}
Among the parameters  $n$,  $\lambda$ and the data $\bar \Gamma_0$, $\bar U_0$ and $\bar \Sigma_0$,  only
 three are independent due to \eqref{idcompat1} and \eqref{idcompat2}.
Given $n$ and two parameters,  say $ \bG_0=\bG(0)$ and  $\bU_0=\bU(0)$ with $\bar \Gamma_0 < \bar U_0$,  
we select  $ \lambda > 0$ and the translaton factor $\eta_0$ 
by \eqref{eq:kappa}:
\begin{equation}
 \lambda = \frac{2-n}{2}\left(\frac{\bU_0}{\bG_0} -1 \right), \qquad \kappa = \bG_0\bU_0^{1-n}. \label{eq:lambda_kappa}
\end{equation}
{  
The inequality \eqref{eq:lambda_upperbdd} restricts the data
\begin{equation}
 1 \;<\; \frac{\bU_0}{\bG_0} \;<\; \frac{2-n}{n}.
\end{equation}
Now,  we pick one heteroclinic $\varphi^\star (\eta) = \big ( p(\eta) , q (\eta) , r (\eta) \big )^T $  having the desired asymptotic behavior 
by the procedure described in Section \ref{sec:singling}. 
}
The solution is reconstructed by  inverting \eqref{eq:defpqr} to obtain  $\tv$, $\tg$, $\tu$ and $\ts$,
\begin{equation}
\label{invrel}
  \tv = \frac{1}{n}\Big( p^{-(1-n)}q^{2-n}r^n \Big)^{ \frac{1}{2-n}},\quad
  \tg = \Big( pr^n \Big)^{ \frac{1}{2-n}},\quad
  \ts = \Big( p^{-(1-n)}r^n \Big)^{ \frac{1}{2-n}},\quad
  \tu = \Big( pr^2 \Big)^{ \frac{1}{2-n}}.
\end{equation}
It is instructive to relate $(p,\ q,\ r)$ to the original variables $\big(v(x,t)$, $\gamma(x,t)$, $\sigma(x,t)$, $u(x,t)\big)$. 
Using \eqref{eq:transform} we recover $\big(\bV(\xi),\ \bG(\xi),\ \bU(\xi),\ \bS(\xi)\big)$ which are smooth functions of $\xi$.
Moreover, using \eqref{eq:selfsimilar} and  \eqref{rescale},  the original variables
$v$, $\gamma$, $\sigma$ and $u = v_x$ are reconstructed by the formulas
\begin{equation}
 \begin{split}
 v(x,t) &= \Big(1 + \frac{t}{ \gamma_0 } \Big)^{ \frac{ \lambda n}{2-n} } \;\bV\bigg(x\Big(1 + \frac{t}{ \gamma_0 } \Big)^{\lambda}\bigg),\\
 \gamma(x,t) &= \gamma_0\Big(1 + \frac{t}{ \gamma_0 } \Big)^{1+ \frac{ 2 \lambda}{2-n} } \;\bG\bigg(x\Big(1 + \frac{t}{ \gamma_0 } \Big)^{\lambda}\bigg), \\
 \sigma(x,t) &= \frac{1}{\gamma_0}\Big(1 + \frac{t}{ \gamma_0 } \Big)^{ -1- \lambda(1-\frac{n}{2-n}) } \;\bS\bigg(x\Big(1 + \frac{t}{ \gamma_0 } \Big)^{\lambda}\bigg), \\
 u(x,t) &= \Big(1 + \frac{t}{ \gamma_0 }\Big)^{\frac{2 \lambda}{2-n} } \;\bU\bigg(x\Big(1 + \frac{t}{ \gamma_0 } \Big)^{\lambda}\bigg). \\
 \end{split} \label{eq:overall_transform}
\end{equation}
The focusing behavior of the profiles is readily deduced from the above formulas and is also illustrated in the
Figures \ref{fig:gamma}-\ref{fig:sigma}, obtained from the computed heteroclinic orbit.
The asymptotic behavior of the profiles  $\big( \bar{ V }, \bar{ \Gamma }, \bar{ \Sigma},\bar{ U }  \big)$ 
as $\xi \to 0 $ and $\xi \to \infty$ captures the behavior of the localizing solution
and is listed  in the following proposition:
\begin{proposition} \label{prop:asymptotic}
For the  orbit $\varphi^\star(\eta)$,  the corresponding profile  $\big( \bar{ V }, \bar{ \Gamma }, \bar{ \Sigma},\bar{ U }  \big)$ 
defined in  $[0,\infty)$ by  \eqref{eq:transform} satisfies the properties:
 \begin{enumerate}[(i)]
  \item At $\xi = 0$ it satisfies the boundary conditions
    \begin{equation*}
    \bar{V}(0) = \bar\Gamma_\xi(0) = \bar\Sigma_\xi(0) = \bar{U}_\xi(0)=0, \quad \bar{ \Gamma}(0) = \bar{ \Gamma}_0, \quad \bar{ U}(0) = \bar{ U}_0. \quad 
  \end{equation*}
  \item The asymptotic behavior as $\xi \rightarrow 0$ is given by 
  \begin{align*}
    \bV(\xi) &= \bU_0\xi + O(\xi^3), & \bG(\xi) &= \bG_0 + O(\xi^2), \\
    \bS(\xi) &= \frac{\bU_0^n}{\bG_0}+ O(\xi^2), & \bU(\xi) &= \bU_0 + O(\xi^2).
  \end{align*}
  \item The asymptotic behavior as $\xi \rightarrow \infty$ is given by
  \begin{align*}
    \bV(\xi) &= O(1), &    \bG(\xi) &= O(\xi^{- \frac{1}{1-n}}),\\
	\bS(\xi) &= O(\xi), &   \bU(\xi) &= O(\xi^{- \frac{1}{1-n}}).
  \end{align*}
 \end{enumerate}
\end{proposition}

{  
Indeed the behavior in (i) and (ii) follows readily from the analysis in Section \eqref{sec:hetero} and the Appendix.  To prove (iii):  Note that the heteroclinic orbit 
$\varphi^\star (\eta)$ must lie in the intersection $W^u (M_0) \cap W^s(M_1)$ and thus satisfies as $\eta \to \infty$ the asymptotic expansion \eqref{eq:local1}.
The term multiplying $e^{-\eta}$ decays much faster than the term $e^{-\frac{n}{1-n} \eta}$ and can be ignored. Using \eqref{eq:local1}, together with 
\eqref{equilibria}, \eqref{equivector}, we compute the asymptotic behavior 
$$
p(\eta) \sim e^{- \frac{n}{ 1-n } \eta } \, ,  \quad q(\eta) \to \frac{2-n}{\lambda} \, , \quad r(\eta) \to 1 - \frac{n \lambda}{(2-n)(1-n)} \quad \mbox{ as $\eta \to \infty$} \, .
$$
In turn, via \eqref{invrel},
$$
\tg (\eta) \sim e^{- \frac{n}{ (1-n)(2-n) } \eta } \, , \quad \tu(\eta) \sim e^{- \frac{n}{ (1-n)(2-n) } \eta } \, . \quad \ts(\eta) \sim e^{\frac{n}{ 2-n } \eta }
\quad \mbox{ as $\eta \to \infty$}
$$
and (iii) follows from \eqref{eq:transform} and $\xi = e^\eta$.

Having identified the asymptotic behavior of $\big(  \bar{ \Gamma }, \bar{ \Sigma},\bar{ U }  \big)$ associated to the selected heteroclinic orbit, we can identify the 
 behavior of $\big(\gamma,\sigma,u\big)$.  Using \eqref{eq:overall_transform} and Proposition \ref{prop:asymptotic}(iii),  we obtain:

\medskip
\noindent
(a)  For the strain:
\begin{align*}
 \gamma(0,t) &= \bar \Gamma_0 \left ( 1 + \tfrac{t}{\gamma_0} \right )^{ 1+ \frac{ 2 }{2-n}\lambda } \\
 \gamma(x,t) &\sim \left ( 1 + \tfrac{t}{\gamma_0} \right )^{ 1- \frac{ n }{(2-n)(1-n)}\lambda } |x|^{-\frac{1}{1-n}} \quad \text{as $t \rightarrow \infty$ for $x\ne 0$}.
\end{align*}
The tip of the strain $\gamma(0,t)$ grows at superlinear order and $\gamma(x,t)$, $x\ne 0$, grows in sublinear order as $t \rightarrow \infty$. Recalling that the growth of the uniform shear motion is linear, this observation points to  localization.

\medskip
\noindent
(b) For the strain rate
\begin{align*}
 u(0,t) &=   \bar U_0 \left ( 1 + \tfrac{t}{\gamma_0} \right )^{ \frac{ 2 }{2-n}\lambda } \\
 u(x,t) &\sim \left ( 1 + \tfrac{t}{\gamma_0} \right )^{- \frac{ n }{(2-n)(1-n)}\lambda } |x|^{-\frac{1}{1-n}} \quad \text{as $t \rightarrow \infty$ for $x\ne 0$}.
\end{align*}
 $u(x,t)$ decays to $0$  for $x\ne0$  in the presence of $n>0$ as $t \rightarrow \infty$. 

%

\medskip
\noindent
(c) For the stress
\begin{align*}
 \sigma(0,t) &= \frac{\bU_0^n}{\bG_0}  \frac{1}{\gamma_0}  \left ( 1 + \tfrac{t}{\gamma_0} \right )^{ -1-  2 \lambda \frac{1 - n}{2-n}}   ,\\
 \sigma(x,t) &\sim \left ( 1 + \tfrac{t}{\gamma_0} \right )^{ -1 + \frac{ n }{2-n}\lambda } |x| \qquad \text{as $t \rightarrow \infty$ for $x\ne 0$}.
\end{align*}
As time proceeds, the stress collapses at the origin quicker to $0$ than at points $x \ne 0$.
}


\section{Numerical Results} \label{sec:numerical_const}
In this section we construct numerically the heteroclinic orbit  proposed in Section  \ref{sec:hetero} and schematically presented in Figure  \ref{fig:hetero}. 
To this end we numerically compute solutions of the dynamical system  \eqref{eq:p}-\eqref{eq:r} 
using appropriate routines from the odesuite of MATLAB.  Care has to be exercised to deal with the known difficulties in numerically computing
heteroclinic orbits. 

The first task is to compute the collection of heteroclinic orbits joining $M_0$ and $M_1$. Such heteroclinics appear as intersections 
of  $W^u(M_0)$, the unstable manifold  of the equilibrium $M_0$, and $W^s(M_1)$, the
stable manifold of the equilibrium $M_1$.  We capture this intersection numerically by using a shooting argument. We start with initial conditions 
near $M_1$ at the directions spanned by the stable eigenvectors of the equilibrium  $M_1$. 
Such data lie  in the stable manifold $W^s(M_1)$; we  compute numerically the solution of \eqref{eq:p}-\eqref{eq:r} solving backwards in $\eta$.
The orbits that tend near the equilibrium $M_0$ are retained and form a two dimensional manifold of heteroclinics, which comprises
the intersection ${W}^u(M_0) \cap W^s(M_1)$. The result of the computation is depicted in Fig. \ref{fig:2manifold}. 

 \begin{figure}[ht]
  \centering
  \subfigure[ Numerical plot of the surface formed by heteroclinic orbits (blue curves) comprising $W^s(M_1) \cap W^u (M_0)$.]{
  \includegraphics[scale=0.6]{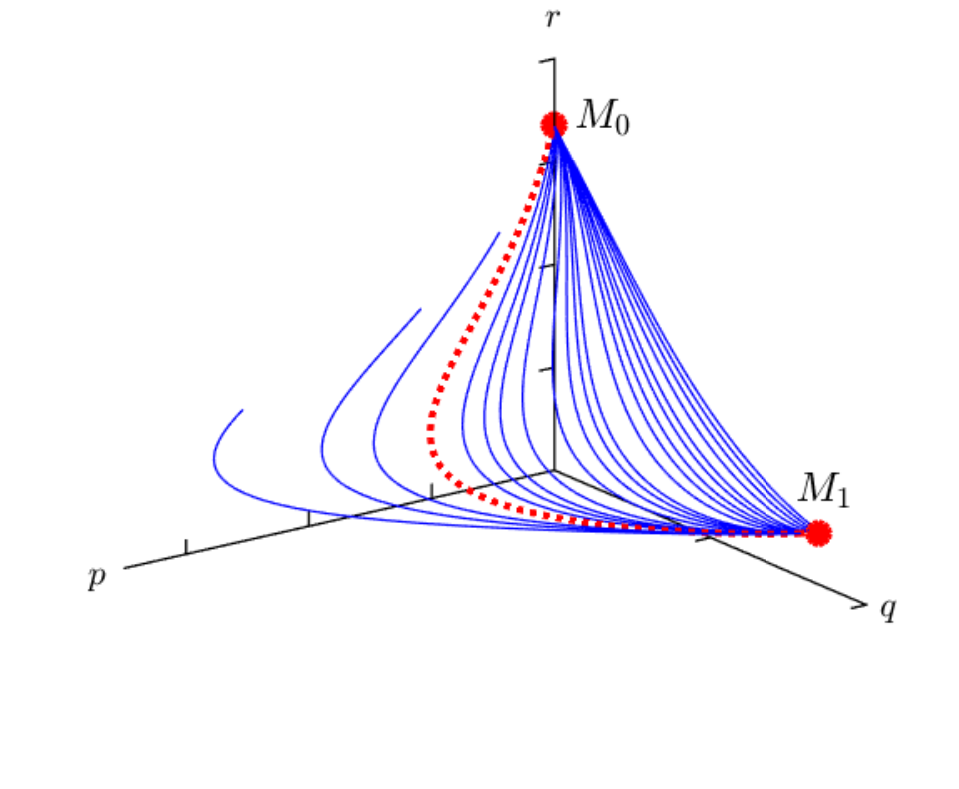} \label{fig:2manifold}
  }
   \subfigure[ A zoom around $M_0$ depicting the intersection of the surface of heteroclinics (blue curves)
with the part of the unstable manifold $W^u (M_0)$ associated to the directions $\vec{X}_{02}$
and $\vec{X}_{03}$ (orange curves). ]{
  \includegraphics[scale=0.6]{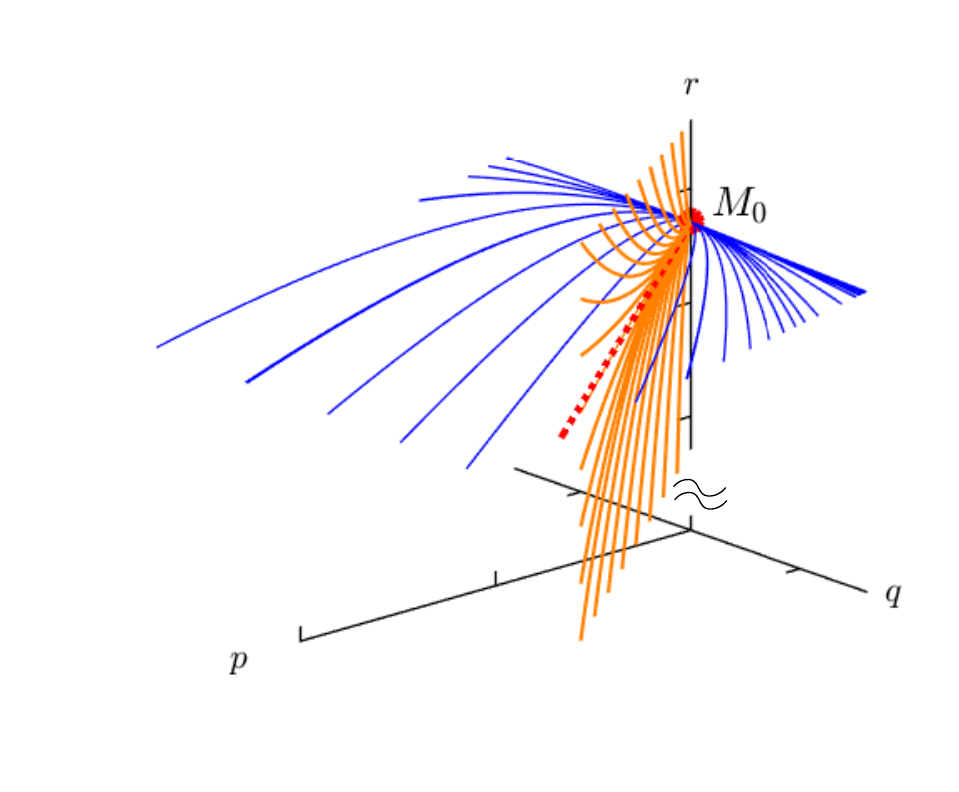} \label{fig:intersection}
  }
  \caption{ Numerical plot of the heteroclinic orbit (red dotted line).}
  \end{figure}

It is expected from the asymptotic behavior in  formula \eqref{eq:local} that the heteroclinics in this manifold will be 
 separated to those approaching $M_0$ along the eigenvector $\vec{X}_{01}$ (the blue curves) and one single orbit
approaching in the direction $\vec{X}_{02}$. To capture computationally the latter, we proceed as follows. Starting now
near $M_0$ we compute the solutions of \eqref{eq:p}-\eqref{eq:r} with initial data close to $M_0$ in the tangent space spanned by $\vec{X}_{02}$
and $\vec{X}_{03}$. We note that the unstable manifold ${W}^u(M_0)$ is the whole space and splits into the blue curves 
(forming the heteroclinic surface) and the orange curves in the transversal directions. This second computation captures the orange curves.
When the orbits reverse directionality the computation captures the (red) heteroclinic emanating in the direction $\vec{X}_{02}$;  see Figure \ref{fig:intersection}
for a zoom of this computation.

%

  \begin{figure}[ht]
  \subfigure[Detected orbit] {
  \hspace{-0.5cm}
  \includegraphics[scale=0.8]{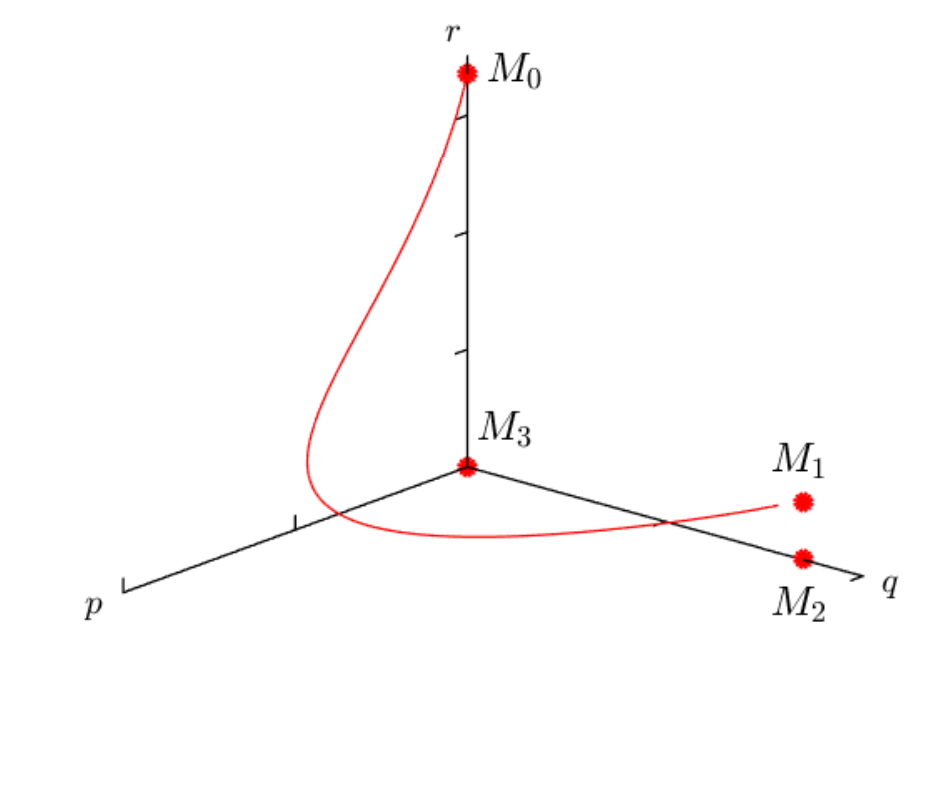} \label{fig:target_orbit_a}
  }\hspace{-0.75cm} \ \ \
  \subfigure[Zoom-in of orbit near $M_0$] {
  \includegraphics[scale=0.75]{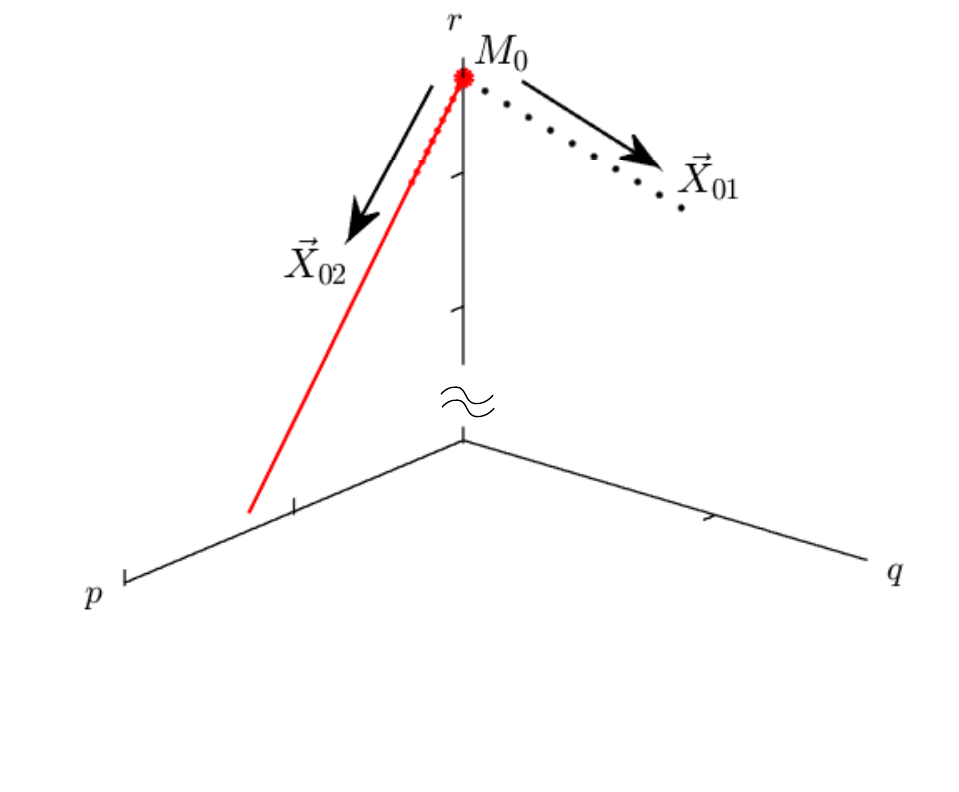} \label{fig:target_orbit_b}
  }
\caption{Simulation of the heteroclinic orbit ($n=0.3$ and $\lambda =2$).} 
\end{figure}

Having computed  the heteroclinic orbit  $(p, q, r)$ of the dynamical system \eqref{eq:p}-\eqref{eq:r}, our next task is to introduce the computed
solution to the transformation \eqref{eq:overall_transform} to produce the numerically computed form of the associated self-similar
solution. The computed solution  is presented  in Figures \ref{fig:gamma}-\ref{fig:sigma}. 
The graph is presented in terms of the original variables  $v(x,t)$, $\gamma(x,t)$, $u(x,t)$ and $\sigma(x,t)$ 
depicting the solution at various instances of time.
From Figures \ref{fig:gamma} - \ref{fig:sigma}, one sees that  initially all the variables are small perturbations of the corresponding uniform shear solutions  \eqref{eq:uss}. 
These nonuniformities then are amplified and eventually  develop a shear band. 
The strain and strain rate exhibit sharp peaks localized around the origin and the velocity field becomes a step like function, 
with the transition again occurring around the origin. The stress field initially assumes finite values, 
but as the shear band forms it collapses to computable zero at the center of the band.

\begin{figure}[ht]
  \centering
  \subfigure[strain ($\gamma$)]{
  \includegraphics[scale=0.3]{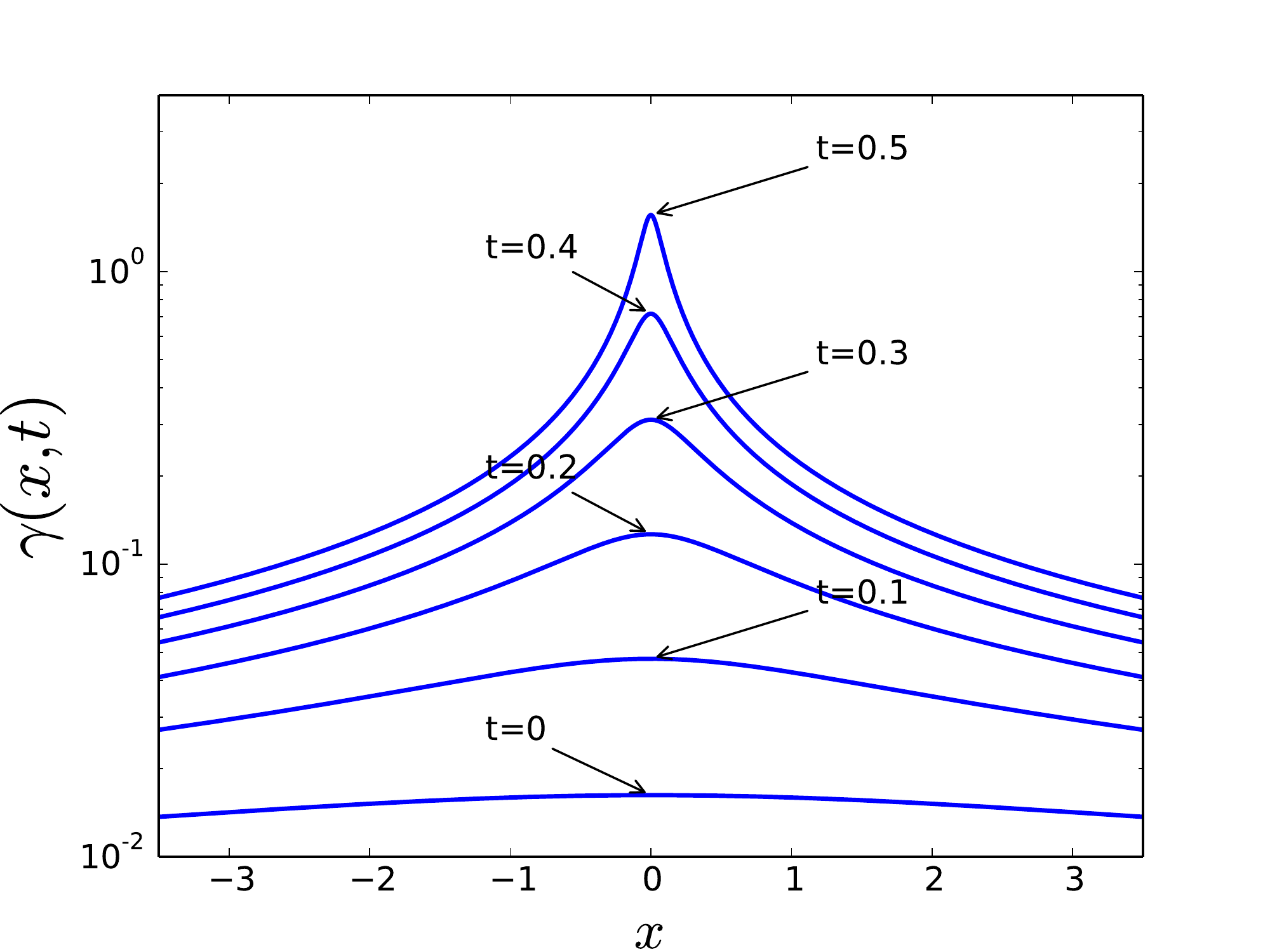} \label{fig:gamma}
  }
  \subfigure[strain rate ($u$)]{
  \includegraphics[scale=0.3]{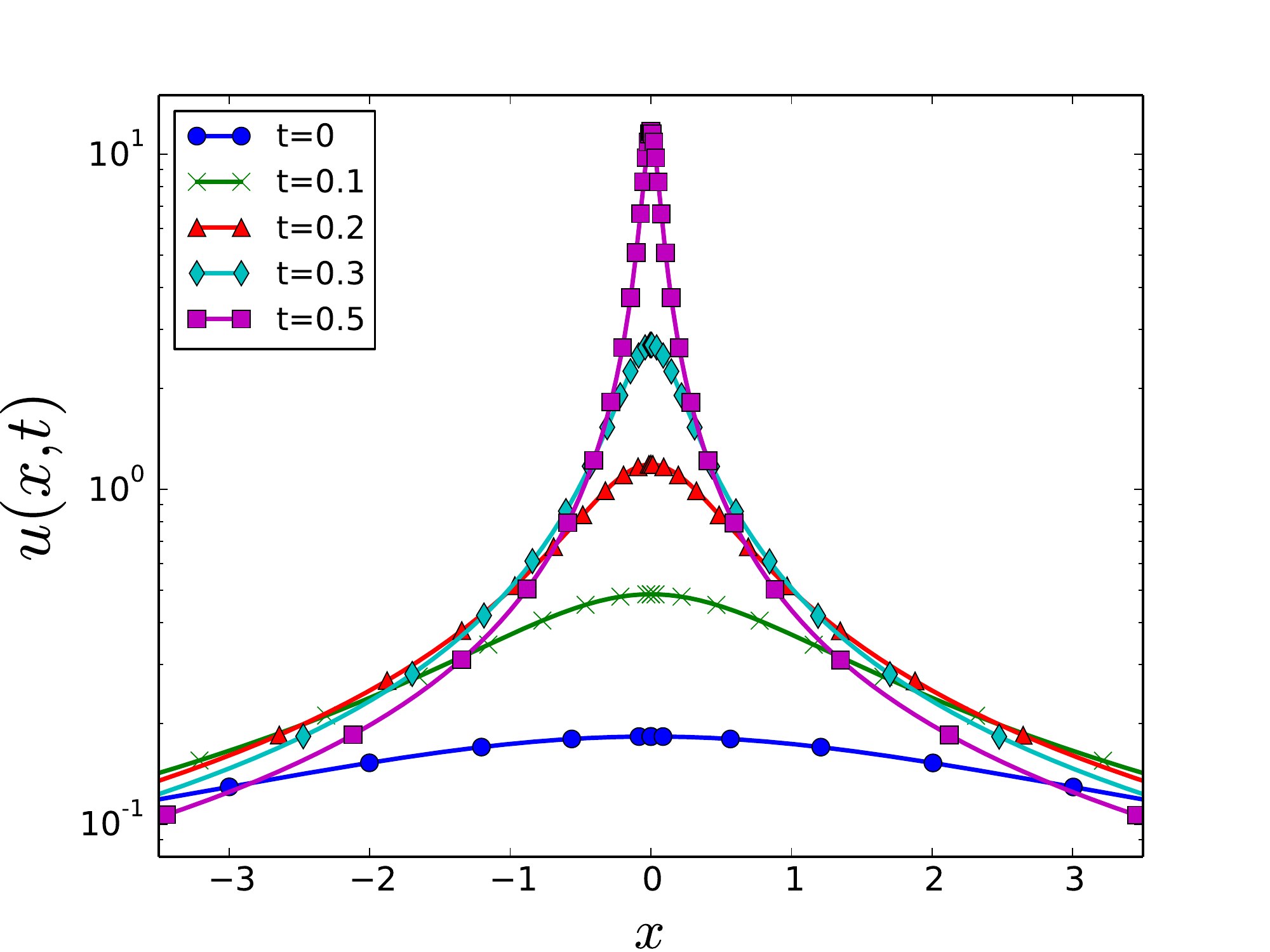} \label{fig:u}
  }
  \subfigure[velocity ($v$)] {
  \includegraphics[scale=0.3]{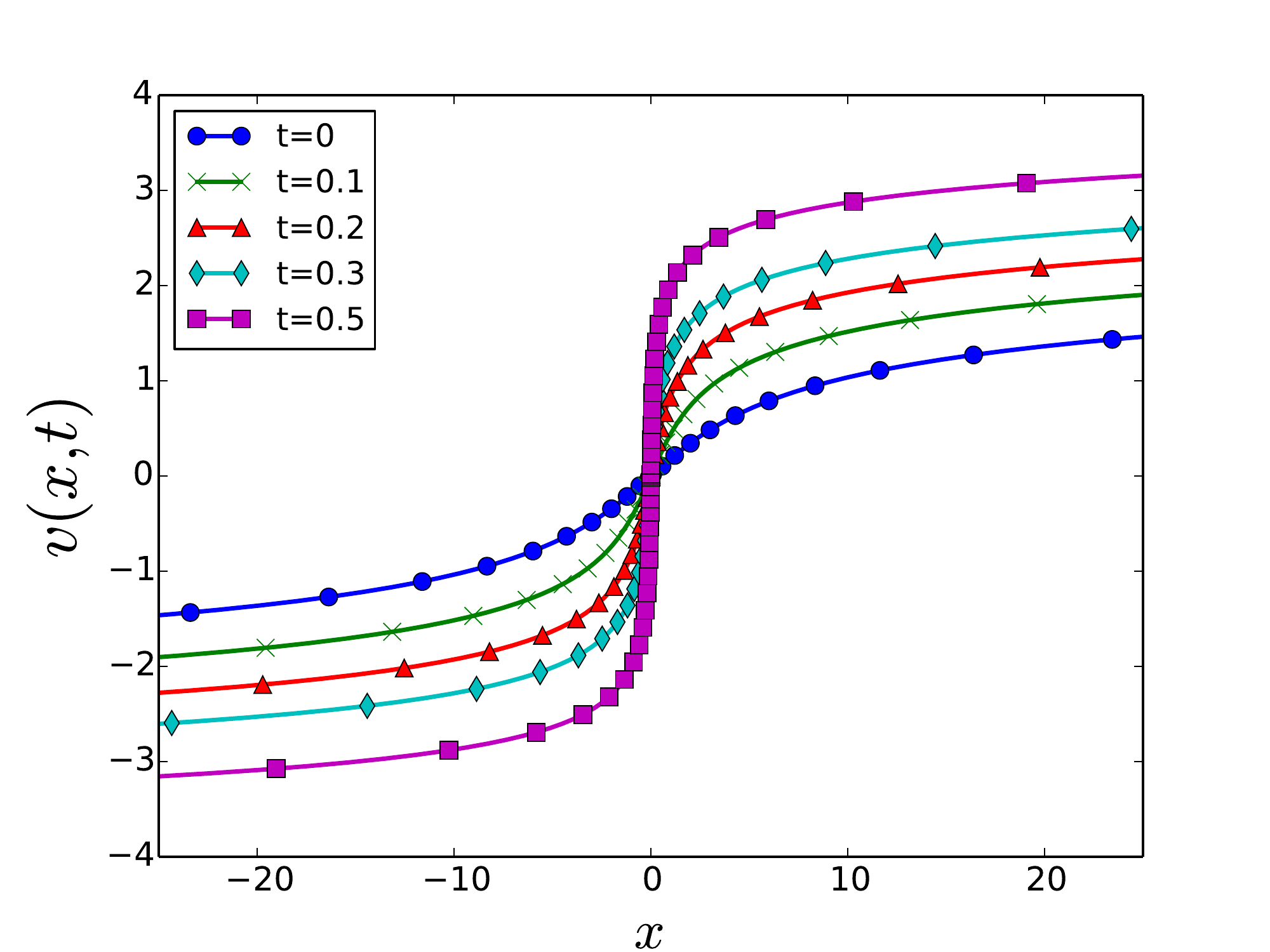} \label{fig:v}
  }
  \subfigure[stress ($\sigma$)] {
  \includegraphics[scale=0.3]{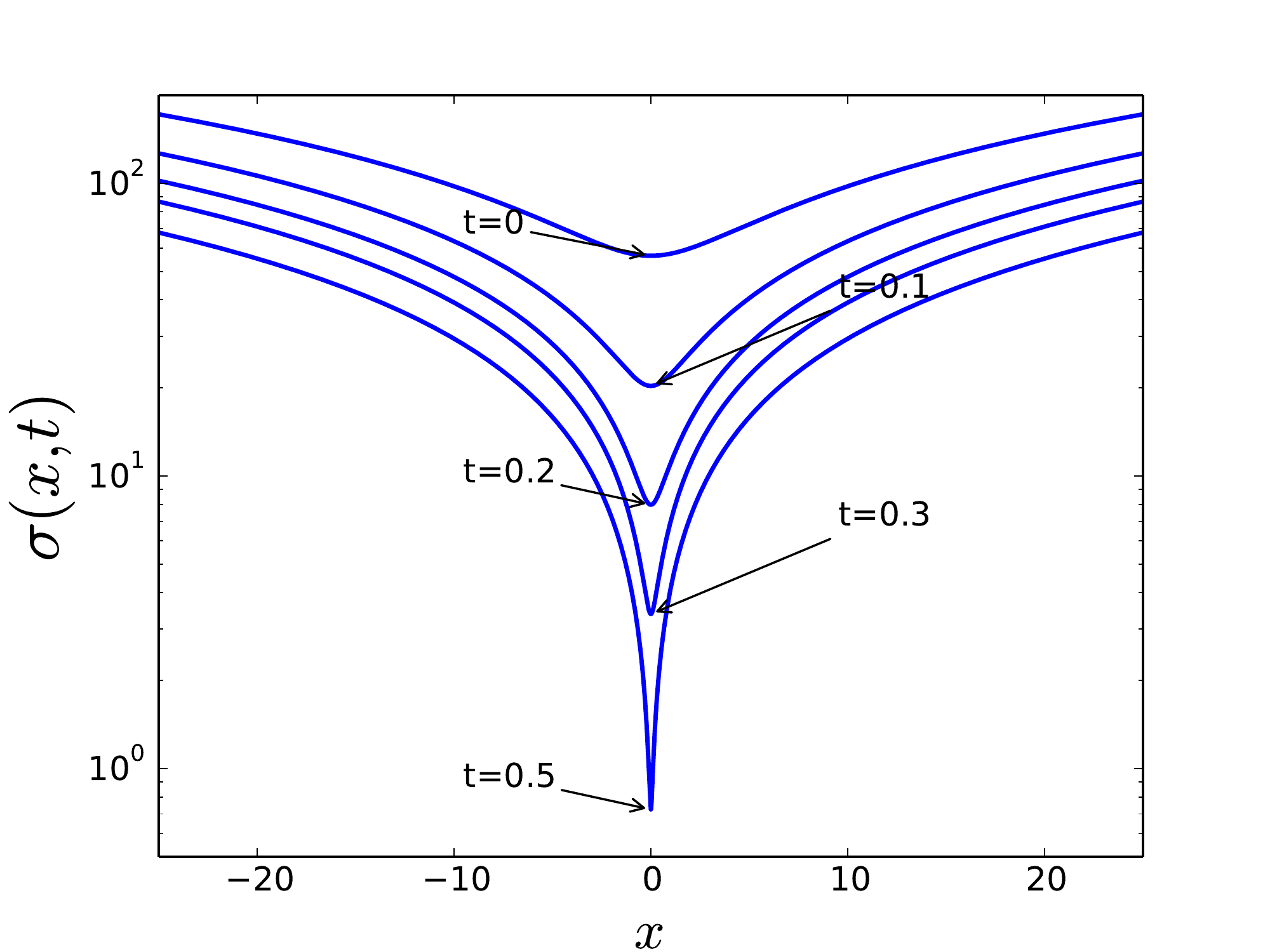} \label{fig:sigma}
  }
\caption{The localizing solution, for $n=0.05$ and $\lambda =10$, sketched in the original variables $v,\ \gamma, \ u, \ \sigma$.
 All graphs except $v$ are in logarithmic scale.} 
\end{figure}

\section{Discussions and Conclusions}

We have considered a model for shear deformation of a viscoplastic material, with a yield relation  depending the plastic strain $\gamma$ and the strain rate $\gamma_t$,
modeling a viscoplastic material exhibiting strain softening and strain rate sensitivity,  in the form of inverse and power law respectively. 
The effects of those two mechanisms were first assessed at the linearized level.  The uniform shearing solutions are stable for $n > 1$, see \cite{tzavaras_plastic_1986},
and we considered the range $0<n<1$. 
In this parameter range, strain rate sensitivity alters the nature of instability from {\it Hadamard instability} to {\it Turing instability}. 

Then we turned into explaining the emergence of a coherent structure in the nonlinear regime.
An effective equation is derived which captures the combined effect of the competition between strain-rate hardening and strain-softening.
We analyzed a class of focusing self-similar solutions exploiting the invariance properties of the underlying model.
This offers a quantitative way for describing the development of singularity beyond the linearized level. 
The associated localized solution arises as a heteroclinic orbit of an induced autonomous dynamical system.
For a certain range of the parameters, the existence of such a heteroclinic orbit was shown numerically. The associated profiles
exhibit coherent localizing structures that have the morphology of shear bands. The result hinges on a computational construction of 
the  heteroclinic orbit. 
The proof of existence of the heteroclinic orbit requires advanced results from dynamical systems, the geometric singular perturbation
theory, and is presented in \cite{LT16}.

\subsection*{Appendix : Proof of \eqref{eq:asymp_alpha}} \label{sec:appendix}

To establish the claim \eqref{eq:asymp_alpha}, the starting point is the  boundary conditions \eqref{eq:barcond2} and \eqref{eq:barcond1}. 
We compute a Taylor expansion of the variables $p(\log\xi)$, $q(\log\xi)$ and $r(\log\xi)$ near $\xi=0$. Initial conditions \eqref{eq:barcond1} imply an asymptotic behavior near $\xi=0$ of the form
\begin{align*}
 \bV(\xi) &= \bU(0)\xi + o(\xi^2),\\
 \bG(\xi) &= \bG(0) + \frac{\xi^2}{2} \bG_{\xi\xi}(0) + o(\xi^2), \\
 \bS(\xi) &= \bS(0) + \frac{\xi^2}{2} \bS_{\xi\xi}(0) + o(\xi^2), \\
 \bU(\xi) &= \bU(0) + \frac{\xi^2}{2} \bU_{\xi\xi}(0) + o(\xi^2).
\end{align*}
The values of the variables and their derivatives evaluated at $\xi=0$ can be inferred by differentiating system \eqref{eq:barsysv}-\eqref{eq:barsyss} repeatedly. 
After a straightforward calculation, we find
\begin{align*}
 \bar{V}(0) &=0, \quad
 \bar{V}_\xi(0) = \bar{U}(0)=\bG(0)\Big(1+ \frac{2 \lambda}{2-n}\Big), \quad
 \\
 \bar{ \Sigma }(0) &= \frac{ \bar{U}(0)^n }{ \bar{ \Gamma }(0) }, \quad
 \bar{\Gamma}_{\xi}(0) = \bar{\Sigma}_{\xi}(0) = \bar{U}_\xi(0) = 0, 
 \end{align*}
 and
 \begin{align*}
 \bar{\Gamma}_{\xi\xi}(0) &=-\big(\bar\Gamma(0)^{2}\bU(0)^{1-n}\big) \frac{ \lambda }{2-n} \frac{ \lambda }{1-n} \Bigg(\frac{1}{ \frac{1}{2} - \frac{n}{1-n} \frac{ \lambda }{r_0}}\Bigg) \frac{1}{ \lambda }, \\
 \bar{U}_{\xi\xi}(0) &=- \big(\bar\Gamma(0)^{2}\bU(0)^{1-n}\big)\frac{ \lambda }{2-n} \frac{ \lambda }{1-n} \Bigg(\frac{r_0 + {2 \lambda }}{ \frac{1}{2} - \frac{n}{1-n} \frac{ \lambda }{r_0}}\Bigg) \frac{1}{ \lambda }.
\end{align*}
Then, the Taylor expansion of $p(\log\xi)$ near $\xi=0$ gives 
\begin{align*}
 {p}(\log\xi) &= \frac{ \tilde \gamma}{ \tilde \sigma} = \frac{ \xi^{ \frac{2}{2-n} } \, \bar\Gamma }{ \xi^{ \frac{2}{2-n} } \, \bar\Sigma } = \xi^2 \frac{ \bar\Gamma }{ \bar\Sigma }
 = \xi^2 \frac{ \bar\Gamma(0) }{ \bar\Sigma(0) } + o(\xi^2) = \bar\Gamma(0)^{2}\bU(0)^{-n}\xi^2 + o(\xi^2) \, ,
\end{align*}
{and the Taylor expansion of $q(\log\xi)$  is }
\begin{align*}
 {q}(\log\xi) &= n\frac{ \tilde v}{ \tilde \sigma} = n\frac{ \xi^{ \frac{n}{2-n} } \, \bar V }{ \xi^{ -1+\frac{n}{2-n} } \, \bar\Sigma } = n\xi \frac{ \bar V }{ \bar\Sigma }= n\bigg[\xi \frac{ \bar V(0) }{ \bar\Sigma(0) } + \xi^2 \frac{\bar{ V }(0)}{\bar{ \Sigma }(0)} \bigg(\frac{\bar{ V }_\xi(0)}{\bar{ V }(0)} - \frac{\bar{ \Sigma }_\xi(0)}{\bar{ \Sigma }(0)}\bigg) + o(\xi^2)\bigg]\\
 &=n\bar\Gamma(0)\bU(0)^{1-n}\xi^2 + o(\xi^2). 
\end{align*}
{Finally, notice first that
\begin{align}
 {r}(\log\xi)&= \frac{ \tilde u}{ \tilde \gamma} = \frac{ \xi^{ \frac{2}{2-n} } \, \bar U }{ \xi^{ \frac{2}{2-n} } \, \bar\Gamma }  = \frac{ \bar U }{ \bar\Gamma } \quad \text{and thus} \quad
 \lim_{\xi \rightarrow0} r(\log\xi) = \frac{\bU(0)}{\bG(0)} = 1 + \frac{2 \lambda}{2-n} = r_0. \label{eq:lambda}
\end{align}
The Taylor expansion of $r(\log\xi)-r_0$ near $\xi=0$ yields }
\begin{align*}
 {r}(\log\xi) -r_0 &= \frac{ \bar U }{ \bar\Gamma } - \frac{ \bar{U}(0)}{ \bar{ \Gamma }(0)} 
 = \xi \frac{\bar{ U }(0)}{\bar{ \Gamma }(0)} \bigg(\frac{\bar{ U }_\xi(0)}{\bar{ U }(0)} - \frac{\bar{ \Gamma }_\xi(0)}{\bar{ \Gamma }(0)}\bigg) \\
 &+ \frac{1}{2}\xi^2\bigg[ \frac{\bar{U}_{\xi\xi}(0)}{ \bar{ \Gamma }(0)} - 2 \frac{ \bar{U}_\xi(0)\bar \Gamma_\xi(0)}{\bar \Gamma^2(0)} + \bar{U}(0) \bigg(- \frac{ \bar\Gamma_{\xi\xi}(0) }{\bar \Gamma^2(0)} + 2 \frac{\big(\bar \Gamma_\xi(0)\big)^2}{ \bar \Gamma^3(0) }\bigg) \bigg] + o(\xi^2)\\
 &=-\big(\bar\Gamma(0)\bU(0)^{1-n}\big)\frac{ \lambda }{2-n} \frac{ \lambda }{1-n} \Bigg(\frac{1}{ \frac{1}{2} - \frac{n}{1-n} \frac{ \lambda }{r_0}}\Bigg)\xi^2 + o(\xi^2) \, .
\end{align*}
From $\log\xi = \eta$, we conclude that
\begin{align*}
   &e^{-2\eta}\left[\begin{pmatrix}
   p(\eta) \\ q(\eta) \\ r(\eta)
  \end{pmatrix}
  - M_0\right]
  \rightarrow {\bar\Gamma(0)\bU(0)^{1-n}}\vec{X}_{02}, \quad \text{as $\eta \rightarrow -\infty$. }
\end{align*}
Now \eqref{eq:barcond2} specifies $\bG(0)$ and in view of \eqref{eq:lambda}, 
\begin{equation}
\kappa = \bar\Gamma_0\bU_0^{1-n}, \quad \text{where $\bU_0 = \bG_0 \Big(1+ \frac{2 \lambda}{2-n}\Big)$}. \label{eq:kappa}
\end{equation}
Notice that $\bG_0$ has the role of determining a translation factor of the curve. 

\begin{remark}
 Let $\big( \bar{ \Gamma }, \bar{ V }, \bar{ \Sigma} \big)$ be a solution of \eqref{eq:barsysv}-\eqref{eq:barsyss} with \eqref{eq:barcond2},\eqref{eq:barcond1} and  $\varphi^\star(\eta)$ be the corresponding heteroclinic orbit of \eqref{eq:p}-\eqref{eq:r}. Further, let $\big( \bar{ \Gamma }_A, \bar{ V }_A, \bar{ \Sigma}_A \big)$ be the rescaled solution as in \eqref{eq:inv_bar} and $\varphi^\star_A(\eta)$ be the corresponding heteroclinic orbit. Then $\varphi^\star_A(\eta) = \varphi^\star(\eta+\eta_0)$ with $\eta_0 = \log A$.

  Let $\eta_0 = \log A$. Then
%
%
\begin{align*}
  \tilde{v}_A\big(\log\xi\big) &:=~~~~\xi^{ \frac{n}{2-n}} \,\bar{V}_A(\xi)=~~~~(A\xi)^{ \frac{n}{2-n}} \,\bar{V}(A\xi)=\tilde{v}\big(\eta+\eta_0\big), \\  
  \tilde{ \gamma}_A\big(\log\xi\big) &:=~~~~\xi^{ \frac{2}{2-n}} \,\bar{ \Gamma}_A(\xi)=~~~~(A\xi)^{ \frac{2}{2-n}} \,\bar{ \Gamma}(A\xi)=\tilde{ \gamma}\big(\eta+\eta_0\big), \\
  \tilde{ \sigma}_A\big(\log\xi\big) &:=\xi^{ -1+\frac{n}{2-n}} \bar{ \Sigma}_A(\xi)=(A\xi)^{ -1+\frac{n}{2-n}} \bar{ \Sigma}(A\xi)=\tilde{ \sigma}\big(\eta+\eta_0\big)
\end{align*}
and hence $\varphi^\star_A(\eta) = \varphi^\star(\eta+\eta_0)$. 
\end{remark}

\subsection*{Acknowledgements}
Research supported by King Abdullah University of Science and Technology (KAUST).

\bibliographystyle{siam}

\end{document}